\documentclass[a4wide]{article}
 \usepackage{authblk}
\usepackage{float}
\usepackage{amsmath,amssymb} 
 \usepackage{stmaryrd} 
\usepackage{graphicx}
\newcommand{\tn}{\interleave}
\newcommand{\mcT}{\mathcal{T}}
\newcommand{\mcF}{\mathcal{F}}
\newcommand{\mcK}{\mathcal{K}}
\newcommand{\mcN}{\mathcal{N}}
\newcommand{\llangle}{\langle\hspace{-2.5pt}\langle}
\newcommand{\rrangle}{\rangle\hspace{-2.5pt}\rangle}

\newcommand{\IR}{\mathbb{R}}
\numberwithin{equation}{section}
\newtheorem{lemma}{Lemma}[section]

\newtheorem{theorem}{Theorem}[section]
\newtheorem{remark}{Remark}[section]
\newtheorem{corollary}{Corollary}[section]
\newenvironment{proof}{\noindent \newline {\bf Proof.}}
{\hfill \mbox{\fbox{} } \newline}

\date{}

\begin{document}

\title{A Cut Finite Element Method for a Model of Pressure in Fractured Media }

\author[$\star$]{Erik Burman}
\author[$\dagger$]{Peter Hansbo}
\author[$\ddagger$]{Mats G. Larson}
\affil[$\star$]{\footnotesize\it  Department of Mathematics, University College London, London, UK--WC1E  6BT, United Kingdom}
\affil[$\dagger$]{\footnotesize\it  Department of Mechanical Engineering, J\"onk\"oping University, SE-55111 J\"onk\"oping, Sweden}
\affil[$\ddagger$]{\footnotesize\it Department of Mathematics and Mathematical Statistics, Ume{\aa}~University, SE-901\,87 Ume{\aa}, Sweden}

\maketitle

\begin{abstract}
We develop a robust cut finite element method for a model of diffusion
in fractured media consisting of a bulk domain with embedded
cracks. The crack has its own pressure field and can cut through the
bulk mesh in a very general fashion. Starting from a common 
background bulk mesh, that covers the domain, finite element spaces 
are constructed for the interface and bulk subdomains leading to 
efficient computations of the coupling terms. The crack pressure field also uses the
bulk mesh for its representation.
The interface conditions are a generalized form of conditions of Robin
type previously considered in the literature which allows the modeling
of a range of flow regimes across the fracture. The method is robust in
the following way: 1. Stability of the formulation in the full range
of parameter choices; and 2. Not sensitive to the location of the
interface in the background mesh. We derive an optimal order a priori
error estimate and present illustrating numerical examples.
\end{abstract}

\section{Introduction} 
The numerical modelling of flow in fractured porous media is important
both in environmental science and in industrial applications. It is
therefore not surprising that it is currently receiving increasing attention
from the scientific computing community. Here we are interested in
models where the fractures are modelled as embedded surfaces of
dimension $d-1$ in a $d$ dimensional bulk domain. Models on this type
of geometries of mixed dimension are typically obtained by averaging
the flow equations across the width of the fracture and introducing
suitable coupling conditions for the modelling of the interaction with
the bulk flow. Such reduced
modelled have been derived for instance in
\cite{HS74,MaJaRo05,ABH09}. The coupling conditions in these models
typically take the form of a Robin type condition. The physical
properties of the coupling enters as parameters in this interface
condition. The size of these parameters can vary with several orders
of magnitude depending on the physical properties of the crack and of the material in
the porous matrix. This makes it challening to derive methods that
both are flexible with respect to mesh geometries and robust with
respect to coupling conditions. A wide variety of different strategies
for the discretisation of fractured porous media flow
has been proposed in the literature. One approach is to use a method
that allows for nonconforming coupling between the bulk mesh and the
fracture mesh \cite{BNY18}, or even arbitrary polyhedral elements in the bulk mesh in order to be able to mesh the
fractures easily. This latter approach has been developed using
discontinuous Galerkin methods \cite{AntFac19}, virtual
element methods \cite{FK18} and high order hybridised methods \cite{CDF18}.

Herein we will consider an unfitted approach, drawing on previous
work \cite{BCHL15, CLEB16,BuHaLa19b} where flow in fractured porous media was modelled in the
situation where the pressure is a globally continuous function. When
using unfitted finite element methods, the bulk mesh can
be created completely independently of the fractures. Instead the
finite element space is modified locally to allow for discontinuities
across fractures and interface conditions are typically imposed
weakly, or using methods similar to Nitsche's method. For other
recent work using unfitted methods we refer to \cite{KMR19}, where a
stabilized Lagrange multiplier method is considered for the interface
coupling and \cite{CO19} where a mixed method is considered for the
Darcy's equations both in the bulk and on the surface.

The upshot here, compared to \cite{BuHaLa19b} is that the
pressure in the crack has its own pressure, allowing for the accurate
approximation of problems where the pressure is discontinuous between
the bulk and the fracture, and that the interface conditions are
imposed in a way allowing for the full range of parameter values in
the Robin condition, without loss of stability. We use the variant of the
interface modelling
considered in \cite{MaJaRo05}, that was also recently applied for the
numerical modelling in \cite{AntFac19}. In these models we may obtain 
a wide range of parameter values in the interface condition and we
therefore develop a method which handle the full range of values and
produces approximations with optimal order convergence. The approach is inspired by the work of Stenberg \cite{JuSt09} and may be viewed as a version of the Nitsche method that 
can handle Robin type conditions and which converges to the standard Nitsche 
method when the Robin parameter tends to infinity. Previous applications of this 
approach in the context of fitted finite elements include \cite{HaHa04} and 
\cite{Be16}.

The finite element spaces are constructed starting from a standard mesh 
equipped with a finite element space. For each geometric domain (subdomains 
and  interface) we mark all the elements intersected by the domain and then we restrict the finite element space to that set to form a finite element space for each domain. This procedure leads to cut finite elements and we use stabilization to ensure that the resulting form associated with the method is coercive and that 
the stiffness matrix is well conditioned.  The stabilization is of face or ghost penalty 
type \cite{BH12,BuHaLa15,LaZa19} , and is added both to the bulk and interface spaces.  Previous related work work on cut finite element methods include the interface problem \cite{HaHa02};
overlapping meshes \cite{HanHanLar}; coupled bulk-surface problems 
\cite{BurHanLarZah16,BuHaLa19,BuHaLa19b} and \cite{GroOlsReu15}; mixed dimensional problems \cite{BurHanLarLar18}, and surface partial 
differential equations \cite{ BuHaLa15,OlsReu17}.  For a general introduction to cut finite element methods we refer to \cite{BCHL15}.

The outline of the paper is as follows: In Section 2 we introduce the model problem 
and discuss the relation between our formulation of the interface conditions and previous work; in Section 3 we formulate the cut finite element method; in Section 
4 we prove the basic properties of the formulation and in particular an optimal order a priori error estimate which is uniform in the full range of interface parameters; and in Section 5 we present numerical results.

\section{The Model Problem}

\subsection{Governing Equations}
Let $\Omega$ be a convex polygonal domain in $\IR^d$, 
$d=2$ or $3$, with boundary $\partial \Omega$ and 
exterior unit normal $n$. Let $\Gamma$ be a smooth embedded 
interface in $\Omega$, which partitions $\Omega$ into 
two subdomains $\Omega_1$ and $\Omega_2$ with exterior unit normals 
$n_1$ and $n_2$. We assume that
$\Gamma$ is a closed surface without boundary residing 
in the interior of $\Omega$, more precisely we assume that there 
is $\delta_0>0$ such that the distance between $\Gamma$ and 
$\partial \Omega$ is larger than $\delta$. We consider for simplicity 
the case with homogeneous Dirichlet conditions on $\partial \Omega$.

The problem takes the form: find 
$u_i:\Omega_i \rightarrow \IR$ and $u_\Gamma :\Gamma\rightarrow \IR$
such that 
\begin{alignat}{3}
-\nabla \cdot A_i \nabla u_i &= f_i& \qquad &\text{in $\Omega_i$}\label{eq:lapbelt1}
\\  \label{eq:lapbelt}
-\nabla_\Gamma \cdot A_\Gamma  \nabla_\Gamma u_\Gamma &= f_\Gamma 
+ \llbracket n \cdot A \nabla u \rrbracket&\qquad &\text{on $\Gamma$}
\\ \label{eq:interface}
n \cdot A \nabla u + B (u - u_\Gamma) &=0& \qquad &\text{on $\Gamma$}
\\
u &= 0 & \qquad  &\text{on $\partial \Omega$}\label{eq:lapbelt2}
\end{alignat}
Here the jump (or sum) of the normal fluxes is defined by
\begin{equation}
\llbracket n \cdot A \nabla v \rrbracket = \sum_{i=1}^2 n_i \cdot A_i \nabla v_i,
\end{equation}
In the interface condition (\ref{eq:interface}), $B$ is a $2\times 2$ symmetric 
matrix with eigenvalues $\lambda_i$ such that 
$\lambda_i \in [0,\infty)$ 
and we used  the notation 
\begin{equation}
n \cdot A \nabla v
=
\left[
\begin{matrix}
n_1 \cdot A_1 \nabla v_1 
\\
n_2 \cdot A_2 \nabla v_2
\end{matrix}
\right],\qquad 
v - v_\Gamma =
\left[
\begin{matrix}
v_1 - v_\Gamma
\\
v_2 - v_\Gamma
\end{matrix}
\right]
\end{equation}
and thus in component form (\ref{eq:interface}) reads
\begin{equation}
\left[
\begin{matrix}
n_1 \cdot A_1 \nabla u_1 
\\
n_2 \cdot A_2 \nabla u_2
\end{matrix}
\right]
+
B
\left[
\begin{matrix}
u_1 - u_\Gamma
\\
u_2 - u_\Gamma
\end{matrix}
\right]
=
\left[
\begin{matrix}
0
\\
0
\end{matrix}
\right]
\end{equation}
The coefficients $A_1$, $A_2$, are smooth uniformly positive definite symmetric 
$d \times d$ matrices, $A_\Gamma$ is smooth tangential to $\Gamma$ 
and uniformly positive definite on the tangent space of $\Gamma$, so 
that 
\begin{equation}
\sum_{i=1}^2 \| \nabla v_i \|^2_{\Omega_i} 
+ \| \nabla_\Gamma v_\Gamma \|^2_\Gamma 
\lesssim 
\sum_{i=1}^2 (A_i \nabla v_i, \nabla v_i )_{\Omega_i} 
+ (A_\Gamma \nabla_\Gamma v_\Gamma,  \nabla_\Gamma v_\Gamma)
\end{equation}
where $\lesssim$ denotes less or equal up to a constant. Finally we assume 
$f_i\in L_2(\Omega_i)$ and $f_\Gamma \in L_2(\Gamma)$.

\begin{remark}
Several generalizations are possible on the external boundary. For instance, we may let the interface intersect the boundary of $\Omega$. In this case we let 
$\nu$ denote the unit exterior conormal to $\Gamma \cap \partial \Omega$, i.e. $\nu$ is tangent to $\Gamma$ and normal to $\partial \Omega \cap \Gamma$, 
and we assume that $\nu \cdot n \geq c > 0$ for some constant $c$ so
that the interface is transversal to $\partial \Omega$. We may then
enforce the Dirichlet condition $u_\Gamma = g_\Gamma$ on $\partial
\Omega \cap \Gamma$ (see \cite{BHLM19}) or some other standard boundary condition.
\end{remark}

\begin{remark} In practical modeling we may want to take the thickness of the interface inte 
account. Assuming that the permeability matrix in an interface of thickness $t$ takes the form 
\begin{equation}
A|_{U_{t/2}(\Gamma)} = A^e_\Gamma + a^e_{\Gamma} n_\Gamma \otimes n_\Gamma
\end{equation}
where $n_\Gamma$ is a unit normal vector field to $\Gamma$, $U_{t/2}(\Gamma)$ is the set of points with distance less than $t/2$ to $\Gamma$, $v^e$ denotes the extension of a function 
$v$ on $\Gamma$ that is constant in the normal direction, $A_\Gamma$ is the tangential tangential permeability tensor, and finally $a_{\Gamma,n}$ is the permeability across the interface. Also assuming that $f = f_\Gamma^e$ and $u = u^e$ in $U_{t/2}(\Gamma)$, the equation on the interface (\ref{eq:lapbelt}) may be modelled as follows
\begin{alignat}{3}
 \label{eq:lapbelt-thickness}
-\nabla_\Gamma \cdot t A_\Gamma  \nabla_\Gamma u_\Gamma &=  t f_\Gamma 
+ \llbracket n \cdot A \nabla u \rrbracket&\qquad &\text{on $\Gamma$}
\end{alignat}
Note that the last term on the right hand side does not scale with $t$ since it accounts for flow 
into the crack from the bulk domains.
\end{remark}
\begin{remark} \label{Remark_2} We comment on how our interface condition (\ref{eq:interface}) relates to the condition in \cite{MaJaRo05} and later reformulated, see \cite{AntFac19}, in terms of averages and jumps of the bulk fields across the interface.
The interface conditions in \cite{MaJaRo05}, equations (3.18) and (3.19), takes the form 
\begin{align}\label{eq:bc-jaffre-a}
\xi n_1 \cdot A_1 \nabla v_1 - (1-\xi) n_2 \cdot A_2 \nabla v_2 
&= \alpha (v_\Gamma - v_1)
\\ \label{eq:bc-jaffre-b}
\xi n_2 \cdot A_2 \nabla v_2 - (1-\xi) n_1 \cdot A_1 \nabla v_1
&= \alpha (v_\Gamma - v_2)
\end{align}
where $\xi$ and $\alpha$ are parameters. The parameter $\alpha$ is related to 
physical properties of the interface as follows
\begin{equation}
\alpha = \frac{2 a_{\Gamma,n}}{d}
\end{equation}
where $a_{\Gamma,n}$ is the permeability coefficient across the interface 
$\Gamma$ and $d$ is the thickness of the interface, see (3.8) in \cite{MaJaRo05}
In matrix form we obtain 
\begin{equation}
\left[
\begin{matrix}
\xi & \xi - 1
\\
\xi - 1  & \xi
\end{matrix}
\right]
\left[
\begin{matrix}
n_1 \cdot A_1 \nabla v_1 
\\
n_2 \cdot A_2 \nabla v_2
\end{matrix}
\right]
+
\left[
\begin{matrix}
\alpha & 0
\\
0 & \alpha
\end{matrix}
\right]
\left[
\begin{matrix}
v_1 - v_\Gamma
\\
v_2 - v_\Gamma
\end{matrix}
\right]
= 0
\end{equation}
which leads to 
\begin{equation}\label{eq:B-numerics}
B = 
\frac{1}{2\xi - 1}
\left[
\begin{matrix}
\xi &  1 - \xi
\\
1 - \xi & \xi
\end{matrix}
\right]
\left[
\begin{matrix}
\alpha & 0 
\\
0 & \alpha
\end{matrix}
\right]
=
\frac{\alpha}{2\xi - 1}
\left[
\begin{matrix}
\xi &  1 -\xi
\\
1 - \xi  & \xi
\end{matrix}
\right]
\end{equation}
We note that we have the eigen pairs 
\begin{equation}\label{eq:B-eig}
B e_1
= \underbrace{\frac{\alpha}{2\xi - 1}}_{\lambda_1}
e_1, 
\qquad 
B
e_2
=\underbrace{\alpha }_{\lambda_2}
e_2
\end{equation}
with the corresponding eigen vectors defined by
\[
e_1 = \frac{1}{\sqrt{2}} \left[
\begin{matrix}
1
\\
1
\end{matrix}
\right] \quad \mbox{ and } e_2  = \frac{1}{\sqrt{2}} \left[
\begin{matrix}
1
\\
-1
\end{matrix}
\right] 
\]
and thus $B$ is positive definite for $\xi >1/2$, singular for $\xi=1/2$, and indefinite for $\xi<1/2$. It is therefore natural to consider the case when $\alpha>0$ and $\xi>1/2$. \emph{We remark that when $\alpha$ tends to infinity both eigenvalues tend to infinity and when $\xi$ tends to $1/2$ from above one eigenvalue tends 
to infinity. It is therefore important to construct a method which is robust in the 
full range $\lambda_i \in (0,\infty)$ of  possible values for the two eigenvalues}

To see the relation to the formulation of the interface conditions in \cite{AntFac19}  
we first note that we have the expansions
\begin{equation}\label{eq:exp-a}
\left[
\begin{matrix}
n_1 \cdot A_1 \nabla v_1 
\\
n_2 \cdot A_2 \nabla v_2
\end{matrix}
\right]
= 
2^{-1/2} \llbracket  n\cdot A \nabla \rrbracket ~
e_1
+
2^{1/2} \llangle n\cdot A \nabla v \rrangle ~
e_2 
\end{equation}
\begin{equation}\label{eq:exp-b}
\left[
\begin{matrix}
v_1 - v_\Gamma 
\\
v_2 - v_\Gamma
\end{matrix}
\right]
= 
2^{1/2} (\llangle v  \rrangle - v_\Gamma)
~ e_1
+
2^{-1/2} \llbracket v \rrbracket 
~ e_2
\end{equation}
where the jumps and averages of the bulk fields across the the interface 
are defined by
\begin{equation}
\llbracket n \cdot A \nabla v \rrbracket = \sum_{i=1}^2 n_i \cdot A_i \nabla v_i,
\qquad
\llbracket v \rrbracket = v_1 - v_2
\end{equation}
\begin{equation}
\llangle n \cdot A \nabla v \rrangle 
= \frac{1}{2}(n_1 \cdot A_1 \nabla v_1 - n_2 \cdot A_2 \nabla v_2),
\qquad \llangle v \rrangle = \frac{1}{2}( v_1 + v_2 )
\end{equation}
Using the expansions (\ref{eq:exp-a}) and 
(\ref{eq:exp-b}) together with  (\ref{eq:B-eig}) and matching the coefficients associated with each eigenvector we obtain the interface conditions
\begin{gather}
\llbracket n \cdot A\nabla v \rrbracket 
+ \frac{2 \alpha}{2\xi - 1 }  (\llangle v \rrangle - v_\Gamma ) =0
\\
 \llangle n  \cdot A \nabla v \rrangle 
 +
  \frac{\alpha}{2}\llbracket v \rrbracket = 0
\end{gather}
which are precisely the conditions used in \cite{AntFac19}. 
\end{remark}

\subsection{Weak Form}
Define the function spaces
\begin{align}
V&=V_1 \oplus V_2 \oplus V_\Gamma
\\
V_i &=\{ v_i \in  H^1(\Omega_i) : v =0 \text{ on $\partial \Omega \cap \partial \Omega_i$}\}\qquad i =1,2
\\
V_\Gamma &=\{ v_\Gamma \in  H^1(\Gamma) : v =0 \text{ on $\partial \Omega \cap \Gamma$}\}
\end{align}
and let $v\in V$ denote the vector $v=(v_1,v_2,v_\Gamma)$. We will
also use the notation $\tilde V$ for functions $v \in V$ such that
$v_i \in H^{\frac32+\epsilon}(\Omega_i)$, $i=1,2$, and $v_\Gamma \in
H^{\frac32+\epsilon}(\Gamma)$, with $\epsilon >0$. Using partial integration on $\Omega_i$ we obtain
\begin{align*}\nonumber
\sum_{i=1}^2 (f_i,v_i)_{\Omega_i} 
&= \sum_{i=1}^2 (-\nabla \cdot A_i  \nabla  u_i, v_i)_{\Omega_i}
\\
& = \sum_{i=1}^2 ( A_i  \nabla  u_i, \nabla v_i)_{\Omega_i}
- (n_i \cdot A_i \nabla u_i,v_i)_{\partial \Omega_i}
\\
& = \sum_{i=1}^2 ( A_i  \nabla  u_i, \nabla v_i)_{\Omega_i}
- (n_i \cdot A_i \nabla u_i,v_i -v_\Gamma)_{\partial \Omega_i}
- (n_i \cdot A_i \nabla u_i, v_\Gamma)_{\partial \Omega_i}
\\
& = \sum_{i=1}^2 ( A_i  \nabla  u_i, \nabla v_i)_{\Omega_i}
- (n\cdot A \nabla u, v - v_\Gamma)_{\Gamma}
- (\llbracket n \cdot A \nabla u \rrbracket, v_\Gamma)_{\Gamma}
\\
& = \sum_{i=1}^2 ( A_i  \nabla  u_i, \nabla v_i)_{\Omega_i} 
+ (B(u-u_\Gamma), v - v_\Gamma)_{\Gamma}
\\ 
&\qquad
+ (A_\Gamma \nabla_\Gamma u_\Gamma, \nabla_\Gamma v_\Gamma)_{\Gamma}
- (f_\Gamma,v_\Gamma)_\Gamma
\end{align*}
Thus we arrive at the weak problem: find $u=(u_1,u_2,u_\Gamma) \in V$ such that
\begin{align}\label{eq:prob-weak}
\boxed{\mathcal{A}(u,v) = L(v)\qquad \forall v \in V}
\end{align}
where the forms are defined by
\begin{align}\label{eq:form-A}
\mathcal{A}(u,v) 
& = \sum_{i=1}^2 ( A_i  \nabla  u_i, \nabla v_i)_{\Omega_i} 
+ (A_\Gamma \nabla_\Gamma u_\Gamma, \nabla_\Gamma v_\Gamma)_{\Gamma}
+ (B(u-u_\Gamma), v - v_\Gamma)_{\Gamma}
\\ \label{eq:form-L}
L(v) &=\sum_{i=1}^2 (f_i,v_i)_{\Omega_i} + (f_\Gamma,v_\Gamma)_\Gamma
\end{align}

\subsection{Existence and Uniqueness}

Introducing the energy norm 
\begin{align}
\tn v \tn^2 = \sum_{i=1}^2 \| v \|^2_{H^1(\Omega_i)} 
+ \| v_\Gamma \|^2_{H^1(\Gamma)} + \| v - v_\Gamma \|^2_\Gamma
\end{align}
on $V$, we directly find using a Poincar\'e inequality and the 
Cauchy-Schwarz inequality that the form $A$ is coercive and 
continuous 
\begin{align}
\tn v \tn^2 \lesssim \mathcal{A}(v,v), \qquad 
\mathcal{A}(v,w) \lesssim  \tn v \tn \, \tn w \tn
\end{align}
Furthermore, $L$ is a continuous functional on $V$ and it follows from the 
Lax-Milgram Lemma that  there is a unique solution in $V$ to (\ref{eq:prob-weak}).

In the case considered here where $\Gamma$ is a smooth, closed
interface the model problem (\ref{eq:prob-weak}) satisfies the elliptic regularity 
estimate 
\begin{equation}\label{eq:ell-reg}
\boxed{
\| u_1 \|_{H^2(\Omega_1)} + \| u_2 \|_{H^2(\Omega_2)}  
+ \| u_\Gamma \|_{H^2(\Gamma)} 
\lesssim    
\| f_1 \|_{\Omega_1}  + \| f_2 \|_{\Omega_2} + \| f_\Gamma \|_\Gamma}
\end{equation}
This follows in a straightforward manner from standard regularity
theory. First note that since $u_i \in H^1(\Omega_i)$, $i=1,2,$ and
$u_\Gamma \in H^1(\Gamma)$ we have  $B(u - u_\Gamma)\vert_{\Gamma}
\in [H^{\frac12}(\Gamma)]^2$ and using \eqref{eq:interface}, $n \cdot A \nabla u \in
[H^{\frac12}(\Gamma)]^2$. This means that the right hand side of
\eqref{eq:lapbelt} is in $L^2$ and hence $u_\Gamma \in H^2(\Gamma)$ by
elliptic regularity. Considering once again \eqref{eq:interface} we
see that in each subdomain the solution coincides with a single domain
solution with a Robin condition with data in $H^{\frac12}(\Gamma)$ on
$\Gamma$. By the elliptic regularity of the Robin problem we can then
conclude that \eqref{eq:ell-reg} holds.

\section{A Robust Finite Element Method}
\subsection{The Mesh and Finite Element Spaces}
To formulate the finite element method we introduce the following notation:
\begin{itemize}
\item Let $\mcT_{h,0}$ be a quasiuniform mesh on $\Omega$ with mesh 
parameter $h \in (0,h_0]$. Define the active meshes 
\begin{align}
\mcT_{h,i} =  \{ T \in \mcT_{h,0} : T \cap \Omega_i \neq \emptyset \}
\quad i=1,2
,
\qquad 
\mcT_{h,\Gamma} =   \{ T \in \mcT_{h,0} : T \cap \Gamma \neq \emptyset \}
\end{align}
associated with the bulk domains $\Omega_i$, $i=1,2,$ and interface 
$\Gamma$, and the domains covered by the meshes
\begin{equation}
O_{h,i} = \cup_{T\in \mcT_{h,i}} \quad i=1,2, \qquad O_{h,\Gamma} = \cup_{T\in \mcT_{h,\Gamma}}
\end{equation}

\item Let $\mcT_{h,i}(\Gamma) = \{ T \in \mcT_{h,i} : T \cap \Gamma \neq \emptyset \}$ and define $\mcF_{h,i}$ as the set of all interior faces associated with an element in $\mcT_{h,i}(\partial \Omega_i)$. 

\item Let $\mcF_{h,\Gamma}$ be the set of all interior faces in $\mcT_{h,\Gamma}$ and $\mcK_{h,\Gamma} = \{ K = T \cap \Gamma : T \in \mcT_{h,\Gamma} \}$.

\item Let $V_{h,0}$ be the space of continuous piecewise linear functions on $\mcT_{h,0}$ and define 
\begin{equation}
V_{h,i} = V_{h,0}|_{\mcT_{h,i}}\quad i=1,2, \qquad V_{h,\Gamma} 
= V_{h,0}|_{\mcT_{h,\Gamma}}
\end{equation}
and 
\begin{equation}
V_h = V_{h,1} \oplus V_{h,2} \oplus V_{h,\Gamma}
\end{equation}
\end{itemize}

\subsection{Standard  Formulation}
The standard finite element method takes the form: find 
$u_h = (u_{h,1},u_{h,2},u_{h,\Gamma}) \in V_h = V_{h,1} 
\oplus V_{h,2} \oplus V_{h,\Gamma}$ such that 
\begin{align}\label{eq:FEM}
\mathcal{A}^S_h(u_h,v) = L(v) \qquad \forall v \in V_h
\end{align} 
Here the form $\mathcal{A}^S_h$ is defined by 
 \begin{equation}
\mathcal{A}^S_h = \mathcal{A} + s_h
 \end{equation}
where $s_h$ is a stabilization term of the form 
\begin{equation}
s_h = s_{h,1} + s_{h,2} + s_{h,\Gamma}
\end{equation}
with
\begin{equation*}
s_{h,i}(v,w) = \sum_{F \in \mcF_{h,i} } h_F \|\zeta(A_i)\|_{\infty,F}
(\llbracket n\cdot \nabla v\rrbracket,  \llbracket n\nabla w \rrbracket)_{F},
\qquad i=1,2
\end{equation*}
where $\zeta(X)$ denotes the maximum eigenvalue of the matrix $X$,
\begin{align*}
s_{h,\Gamma}(v,w) & = 
\sum_{F \in \mcF_{h,\Gamma} } h_F
                    \|\zeta(A_\Gamma)\|_{\infty,F\cap\Gamma}
                    (\llbracket n\cdot \nabla v \rrbracket, \llbracket
                    n
                    \nabla w \rrbracket)_{\mcF_{h,\Gamma}} \\
& + 
\sum_{T \in \mcT_{h,\Gamma} }h_K^2\|\zeta(A_\Gamma)\|_{\infty,K\cap\Gamma}
(n_\Gamma \cdot \nabla v,  n_\Gamma \cdot \nabla w)_{T \cap \Gamma}.
\end{align*}

\subsubsection{Properties of the Stabilization Terms}

The rationale for the design of the stabilizing terms is that they
improve the stability, while remaining consistent for sufficiently
smooth solutions.

Accuracy relies on the
following consistency property that is immediate from the definitions
above. For any function $v \in H^{\frac32+\epsilon}(O_{h,i})$ there
holds $s_{h,i}(v,w) = 0$ for all $w \in V_{h,i}+H^{\frac32+\epsilon}(O_{h,i})$, $i=1,2$. For any
function $v \in H^{\frac32+\epsilon}(O_{h,\Gamma})$, such that
$n_\Gamma \cdot \nabla v =0$ on $\Gamma$ there holds
$s_{h,\Gamma}(v,w) = 0$ for all $w \in V_{h,\Gamma}+H^{\frac32+\epsilon}(O_{h,\Gamma})$.

The stability properties are well known and we
collect them in the following Lemma.

\begin{lemma}\label{lem:ghost_stab} There are constants such that 
\begin{equation}\label{eq:stab-bulk}
\| \nabla v \|^2_{A_i, O_{h,i} }
\lesssim \| \nabla v \|^2_{A_i,\Omega_i} 
+ \| v \|^2_{s_{h,i}} \qquad i =1,2
\end{equation}
and 
\begin{equation}\label{eq:stab-interface}
\| \nabla_\Gamma v \|^2_{A_\Gamma, O_{h,\Gamma}} \lesssim 
\| \nabla_\Gamma v \|^2_{A_\Gamma,\Gamma} 
+ \| v \|^2_{s_{h,\Gamma}}
\end{equation}
where we introduced the (semi) norm $\| v \|^2_{s_h} = s_h(v,v)$.
\end{lemma}
\begin{proof} See \cite{BH12}, \cite{BuHaLa15}, and  \cite{LaZa19},
  with minor modifications to account for the varying coefficients. 
\end{proof}
\begin{remark}
Observe that the hidden constants in Lemma \ref{lem:ghost_stab} depend
on the variation of the $A_i$ and $A_\Gamma$. 
\end{remark}

\subsection{Robust Formulation}\label{sec:robust}
The stabilizing terms ensure robustness irrespective of the
intersection of the fracture and the mesh. They do not counter
instabilities due to degenerate $B$. Our aim is to design a formulation which is robust in the case when the
eigenvalues of $B$ degenerate. Indeed as we saw above as $\xi$
approaches $1/2$, $\lambda_1$ blows up. For clarity we recall the abstract boundary condition 
\begin{equation}
n\cdot A\nabla v + B (v - v_\Gamma) = 0
\end{equation}
where we now assume that the matrix $B$ is a positive definite symmetric $2\times 2$ matrix with eigenvalues $\lambda_i $ and 
eigenvectors $e_i$, such that $\lambda_i \in  (0,\infty)$ and thus one or both eigenvalues may become very large or small. To handle this situation  we 
instead enforce 
\begin{equation}
B^{-1} n\cdot A\nabla v + (v - v_\Gamma) = 0
\end{equation}
weakly using a modified Nitsche method. This approach was originally developed 
in \cite{JuSt09} where fitted finite element approximation of Robin conditions were considered. 

\paragraph{Derivation of an Alternative Weak Form.} As before we have 
the identity 
\begin{align}
L(v) 
& = \underbrace{\sum_{i=1}^2 ( A_i  \nabla  u_i, \nabla v_i)_{\Omega_i} +
(A_\Gamma \nabla_\Gamma u_\Gamma, \nabla_\Gamma v_\Gamma)_{\Gamma}}_{=:\mathcal{A}_1(u,v)}
- (n\cdot A \nabla u, v - v_\Gamma)_{\Gamma}
\\
&=\mathcal{A}_1(u,v) - (n\cdot A \nabla u, v - v_\Gamma)_{\Gamma}
\end{align}
where we introduced the bilinear form $\mathcal{A}_1$ for brevity. To enforce the interface conditions we proceed as follows
\begin{align*}
L(v) 
&=\mathcal{A}_1(u,v) - (n\cdot A \nabla u, v - v_\Gamma)_{\Gamma}
\\
&=\mathcal{A}_1(u,v) + (n\cdot A \nabla u, B^{-1} (n \cdot A \nabla v) )_{\Gamma} 
\\
&\qquad 
- (n\cdot A \nabla u, B^{-1} (n \cdot A \nabla v) +  (v - v_\Gamma) )_{\Gamma}
\\
&=\mathcal{A}_1(u,v) + (n\cdot A \nabla u, B^{-1} (n \cdot A \nabla v) )_{\Gamma} 
\\
&\qquad 
- (n\cdot A \nabla u, B^{-1} (n \cdot A \nabla v) +  (v - v_\Gamma) )_{\Gamma}
\\
&\qquad 
- (B^{-1} (n \cdot A \nabla u) +  (u - u_\Gamma), n\cdot A \nabla v  )_{\Gamma}
\\
&\qquad 
+ (B^{-1} (n \cdot A \nabla u) +  (u - u_\Gamma) , \tau (B^{-1} (n \cdot A \nabla v) +  (v - v_\Gamma)) )_{\Gamma}
\end{align*}
where the last two terms are zero due to the interface condition and the resulting form on the right hand side is symmetric. Furthermore, $\tau$ is a stabilization parameter (a $2\times 2$ matrix) of the form
\begin{align}\label{eq:tau-def}
\tau = \sum_{i=1}^2 \tau_i e_i \otimes e_i , \qquad \tau_i = \frac{\lambda_i \beta}{ \lambda_i h + \beta} \quad i=1,2
\end{align} 
where $\beta$ is a positive parameter and we recall that $\lambda_i$ and $e_i$ are the eigenvalues and eigenvectors of $B$. The parameter $\beta$ is chosen 
so that 
\begin{equation}\label{eq:tau-beta-condition}
\| n \|^2_{A,\infty,\Gamma} := \sum_{i=1}^2 \| n_i \|^2_{A_i,\infty,\Gamma} \lesssim \beta
\end{equation}
where $\| w \|_{A_i,\infty,\Gamma} := \| A^{\frac12}_i w\|_{\infty, \omega}$ is the $A_i$ weighted 
$L^\infty$ norm over $\Gamma$.
\begin{remark} The choice of $\tau_i$ can be further refined as follows
\begin{align}\label{eq:tau-def-refined}
\tau_i = \frac{\lambda_i \beta_i}{ \lambda_i h + \beta_i} \quad i=1,2
\end{align}
with
\begin{equation}
\sum_{j=1}^2 \| n_j \|^2_{A_j,\infty,\Gamma} |e_{ij}|^2 \lesssim \beta_i
\end{equation}
where $e_i = [e_{i1}\; e_{i2}]^T$. This approach is beneficial in situations where the components of $e_i$ are very 
different and there is a large difference between the $\| n_j
\|^2_{A_j,\infty,\Gamma}$ with $j=1$ and $j=2$.
\end{remark}

\paragraph{The Robust Finite Element Method.} Find $u_h \in V_h$ such that 
\begin{equation}\label{eq:robustFEM}
\boxed{\mathcal{A}^R_h(u_h,v):=\mathcal{A}^R(u_h,v) + s_h(u_h,v) = L(v) \qquad \forall v \in V_h}
\end{equation}
 where 
 \begin{align}\label{eq:Ah}
 \mathcal{A}^R(v,w) &= \mathcal{A}_1(v,w) 
 + (n\cdot A \nabla v, B^{-1} (n \cdot A \nabla w) )_{\Gamma} 
\\
&\qquad 
- (n\cdot A \nabla v, B^{-1} (n \cdot A \nabla w) +  (w - w_\Gamma) )_{\Gamma}
\\
&\qquad 
- (B^{-1} (n \cdot A \nabla v) +  (v - v_\Gamma), n\cdot A \nabla w  )_{\Gamma}
\\
&\qquad 
+ (B^{-1} (n \cdot A \nabla v) +  (v - v_\Gamma) , \tau (B^{-1} (n \cdot A \nabla w) +  (w - w_\Gamma)) )_{\Gamma}.
 \end{align}
It follows by the design of $\mathcal{A}^R$ that for a sufficiently smooth exact
solution $u \in \tilde V$ of the problem
\eqref{eq:prob-weak} 
there holds
\begin{equation}\label{eq:robust_consist}
\mathcal{A}(u,v) = \mathcal{A}^R(u,v) = L(v), \quad \forall v \in (V
\cap H^2(\Omega_1\cup \Omega_2 \cup \Gamma) + V_h.
\end{equation}
As a consequence we immediately get the Galerkin orthogonality
\begin{lemma}
Let $u \in \tilde V$ be the solution of
\eqref{eq:prob-weak} and $u_h \in V_h$ the solution of
\eqref{eq:robustFEM} then there holds
\begin{equation}\label{eq:galortho}
\mathcal{A}^R(u - u_h,v) = s_h(u_h,v)\quad \forall v \in V_h.
\end{equation}
\end{lemma}
\begin{proof}
The proof follows by combining \eqref{eq:robust_consist} and \eqref{eq:robustFEM}.
\end{proof}
\section{Error Estimates}

\subsection{The Energy Norm} We introduce the energy norm 
\begin{equation}\label{eq:disc_energy}
\tn v \tn_h^2 
=
\sum_{i=1}^2 \| \nabla v_i \|^2_{A_i,\Omega_i} 
+ h \|  \nabla v_i \|^2_{A_i,\Gamma}
+ \| v \|^2_{s_{h}}
+ \| \nabla_\Gamma v_\Gamma \|^2_{A_\Gamma,\Gamma}
+ \|v  - v_\Gamma\|^2_{\tau,\Gamma}
\end{equation}
where $\| w \|^2_{\psi,\omega} = \int_\omega \psi w^2$ is the $\psi$ weighted 
$L^2$ norm over the set $\omega$.

\subsection{Interpolation Error Estimates}

We begin by introducing the interpolation operators and derive the basic approximation
error estimates. Then collecting the estimates we show an interpolation error 
estimate in the energy norm \eqref{eq:disc_energy}. Since the stabilization operator acts on the finite element solution
outside its physical domain of definition, we must make sense of the
solution it approximates also outside its physical domain of
definition. We will show below show how this can be done using extensions from
the physical geometry.

\paragraph{The Scott-Zhang Interpolant.} Given a mesh $\mcT_h$ covering a domain $O_h$ and the space of piecewise 
linear continuous finite elements $W_h$, the standard Scott-Zhang interpolation operator $\pi_{h,SZ} : H^1(\Omega_h) \rightarrow W_h$ satisfies the element wise estimate

\begin{align}\label{eq:interpol-SZ}
\| v - \pi_{h,i,SZ} v \|_{H^m(T)} 
\lesssim 
h^{2-m} \| v \|_{H^2(\mcN(T))}, \quad m=0,1
\end{align}
where $\mcN(T)$ is the set of all elements in $\mcT_{h,i}$ that share a node with $T$. 
Note also that the Scott-Zhang interpolant preserves homogeneous boundary conditions exactly. See \cite{ScZh90} for further details.

\paragraph{Bulk Domain Fields.}
It is shown in \cite[Section 2.3, Theorem 5]{stein70} that there is an extension operator $E_i: H^s(\Omega_i)\rightarrow H^s(\IR^d)$, 
not dependent on $s\ge 0$, which is stable in  the sense that 
\begin{equation}\label{eq:stability-ext}
\| E_i v_i \|_{H^s(\IR^d)} \lesssim \| v_i \|_{H^s(\Omega_i)}
\end{equation}
We define the interpolation operator $\pi_{h,i}:H^1(\Omega_i) \rightarrow V_{h,i}$ by
\begin{equation}
\pi_{h,i} v_i = \pi_{h,i,SZ} E_i v
\end{equation}
where $\pi_{h,i,SZ}:H^1(O_{h,i}) \rightarrow V_{h,i}$ is the Scott-Zhang
 interpolant and we recall that $O_{h,i} = \cup_{T \in \mcT_{h,i}} T$ is the 
 domain covered by $\mcT_{h,i}$. We then have the error estimate 
 \begin{equation}\label{eq:interpol-bulk}
\boxed{ \| v_i - \pi_{h,i} v \|_{H^m(\Omega_i)} 
 \lesssim h^{2-m} \| v_i \|_{H^2(\Omega_i)}\quad m=0,1}
 \end{equation}
\begin{proof} Using the notation $\rho_i = v_i - \pi_{h,i} v_i $ we obtain
\begin{align*}
\|\rho_i \|_{H^m(\Omega_i)} 
\lesssim
\| \rho_i \|_{H^m(O_{h,i})} 
\lesssim 
h^{2-m} \| E_i u_i \|_{H^2(O_{h,i})}
\lesssim 
h^{2-m} \| u_i \|^2_{H^2(\Omega_i)}
\end{align*}
where we used the fact that $\Omega_i \subset O_{h,i}$, the interpolation error 
estimate (\ref{eq:interpol-SZ}), and finally the stability (\ref{eq:stability-ext}) of 
the extension operator $E_i$.
\end{proof}

\paragraph{Interface Field.}
Let $p_\Gamma : U_\delta(\Gamma) \rightarrow \Gamma$ be the closest
point mapping from the tubular neighborhood $U_\delta(\Gamma):=\{x:
\mbox{dist}(x,\Gamma) < \delta \}$ to $\Gamma$, 
which is well defined for all $\delta \in (0,\delta_0]$ for some $\delta_0>0$. Define the extension operator $E_\Gamma : L^2(\Gamma) \rightarrow L^2(U_{\delta}(\Gamma))$ by $E_\Gamma v = v \circ p_\Gamma$. Since $\Gamma$ is smooth 
we have the stability estimate
\begin{equation}\label{eq:stability-extension-Gamma}
\| E_\Gamma v_\Gamma \ \|_{H^s(U_\delta(\Gamma))} \lesssim \delta^{1/2} \| v_\Gamma \|_{H^s(\Gamma)}.
\end{equation}
Observe also that since $n_\Gamma \cdot \nabla E_\Gamma v_\Gamma = 0$
by construction, and then assuming $s>3/2$ in
\eqref{eq:stability-extension-Gamma} we see that
\begin{equation}
s_{h,\Gamma}(E_\Gamma v_\Gamma,w) = 0, \forall w \in V_{h,\Gamma}+H^{\frac32+\epsilon}(O_{h,\Gamma}).
\end{equation}

To define the interpolant we first let
\begin{equation}
\mcT_{h,\delta,\Gamma} = \{ T \in \mcT_{h,0} : T\cap U_{\delta}(\Gamma)\neq \emptyset \}, 
\qquad 
O_{h,\delta,\Gamma} = 
\cup_{T\in \mcT_{h,\delta,\Gamma}} T
\end{equation}
for $\delta \in (0,\delta_0/2]$. Then $O_{h,\Gamma} \subset O_{h,\delta,\Gamma}   $ and there are $\delta,\delta'\in (0,\delta_0]$ such that 
$\delta \sim \delta' \sim h$ and 
\begin{equation}\label{eq:sets-inclusions}
U_{\delta}(\Gamma)   \subset O_{h,\delta,\Gamma} \subset U_{\delta'}(\Gamma)  \subset U_{\delta_0}(\Gamma)
\end{equation}
for all  $h \in (0,h_0]$, with $h_0$ small enough. We let 
$V_{h,\delta,\Gamma} = V_{h,0}|_{O_{h,\delta,\Gamma}}$ and define $\pi_{h,\Gamma}:H^1(\Gamma) \rightarrow V_{h,\Gamma}$ by
\begin{equation}
\pi_{h,\Gamma} v_\Gamma =( \pi_{h,\delta,\Gamma,SZ} E_\Gamma v_\Gamma )|_{O_{h,\Gamma}}
\end{equation}
where $\delta \sim h$ and $\pi_{h,\delta,\Gamma,SZ}:H^1(O_{h,\delta,\Gamma}) \rightarrow V_{h,\delta,\Gamma}$ is the Scott-Zhang interpolant. We have the 
error estimate 
 \begin{equation}\label{eq:interpol-interface}
\boxed{ \| v - \pi_{h,\Gamma} v \|_{H^m(\Gamma)} 
 \lesssim h^{2-m} \| v \|_{H^2(\Gamma)}\quad m=0,1}
 \end{equation}
\begin{proof} Using the trace inequality  
\begin{equation}
\| v \|^2_\Gamma 
\lesssim 
\delta^{-1} \| v \|^2_{U_\delta(\Gamma)} 
+ \delta \| \nabla v \|^2_{U_\delta(\Gamma) }
\qquad v \in H^1(U_\delta (\Gamma))
\end{equation}
where the hidden constant is independent of $\delta$, we obtain
\begin{align*}
\|\nabla_\Gamma^m \rho \|^2_{\Gamma}
&\lesssim 
\delta^{-1} \|\nabla^m \rho \|^2_{U_{\delta}(\Gamma)}
+\delta \| \nabla^{m+1} \rho \|^2_{U_{\delta} (\Gamma)}
\\
&\lesssim 
\delta^{-1} \|\nabla^m \rho \|^2_{O_{h,\delta,\Gamma}}
+\delta \| \nabla^{m+1} \rho \|^2_{O_{h,\delta,\Gamma}}
\\
&\lesssim 
\delta^{-1} h^{2(2-m)} \|\nabla^2 E_\Gamma v \|^2_{O_{h,\delta,\Gamma}}
+\delta h^{2(1-m)} \| \nabla^{2} E_\Gamma v \|^2_{O_{h,\delta,\Gamma}}
\\
&\lesssim 
\delta^{-1} h^{2(2-m)} \|\nabla^2 E_\Gamma v \|^2_{U_{\delta'}(\Gamma)}
+\delta h^{2(1-m)} \| \nabla^{2} E_\Gamma v \|^2_{U_{\delta'}(\Gamma)}
\\
&\lesssim 
\delta^{-1} \delta' h^{2(2-m)} \| v \|^2_{H^2(\Gamma)}
+\delta \delta' h^{2(1-m)} \| v \|^2_{H^2(\Gamma)}
\\
&\lesssim 
h^{2(2-m)} \| v \|^2_{H^2(\Gamma)}
\end{align*}
where we used (\ref{eq:sets-inclusions}), the interpolation error estimate (\ref{eq:interpol-SZ}), the stability (\ref{eq:stability-extension-Gamma})  
of the extension operator $E_\Gamma$, and the fact that 
$\delta \sim \delta' \sim h$.
\end{proof}

We define the interpolation operator $\pi_h: V \rightarrow V_h$ as follows
\begin{equation}
\pi_h v = (\pi_{h,1} E_1 v_1, \pi_{h,2} E_2 v_2, \pi_{h,\Gamma} E_\Gamma v_\Gamma )
\end{equation}

\begin{lemma}There is a constant not dependent on the matrix $B$, in the interface condition (\ref{eq:interface}), such that 
\begin{align}\label{eq:interpol}
\boxed{
\tn v - \pi_h v \tn_h \lesssim h \left( \sum_{i=1}^2 \| v_i \|_{H^2(\Omega_i)} + \| v_\Gamma \|_{H^2(\Gamma)} \right)
}
\end{align}
\end{lemma}
\begin{proof} Let $v - \pi_h v = \rho$ be the interpolation error. Using the triangle inequality and (\ref{eq:tau-bounds-c}),
\begin{align*}
\tn \rho \tn_h^2 
&=
\sum_{i=1}^2 \| \nabla \rho_i \|^2_{A_i,\Omega_i} 
+ h \|  \nabla \rho_i \|^2_{A_i,\Gamma}
+ \| \nabla_\Gamma \rho_\Gamma \|^2_{A_\Gamma,\Gamma}
+\|\rho_i - \rho_\Gamma \|^2_{\tau,\Gamma}
+ \| \rho \|^2_{s_h}
\\
&\lesssim
\sum_{i=1}^2 \|  \nabla \rho_i \|^2_{\Omega_i} 
+ h \|   \nabla \rho_i \|^2_{\Gamma}
+ h^{-1} \|\rho_i \|^2_{\Gamma}
+ \| \rho_i \|^2_{s_{h,i}}
\\
&\qquad 
+ \| \nabla_\Gamma \rho_\Gamma \|^2_{\Gamma}
+h^{-1} \| \rho_\Gamma \|^2_{\Gamma}
+ \| \rho_\Gamma \|^2_{s_{h,\Gamma}}
\\
&\lesssim
\sum_{i=1}^2 
\left( \sum_{m=0}^2 h^{2(m-1)}  \| \rho_i \|^2_{H^m(O_{h,i})} \right)
\\
&\qquad 
+ \| \nabla_\Gamma \rho_\Gamma \|^2_\Gamma
+h^{-1} \| \rho_\Gamma \|^2_{\Gamma}  
+ \left( \sum_{m=1}^2 h^{2(m-1)}  \| \rho_i \|^2_{H^m(O_{h,\delta,\Gamma})} \right)
\\
&\lesssim \sum_{i=1}^2 h^2 \| E_i v_i \|^2_{H^2(O_{h,i})} 
+  h \| E_\Gamma v_\Gamma \|^2_{H^2(O_{h,\delta,\Gamma})}
\\
&\lesssim \sum_{i=1}^2 h^2 \| v_i \|^2_{H^2(\Omega_i)} 
+  h \delta' \| v_\Gamma \|^2_{H^2(\Gamma)}
\end{align*}
with $\delta' \sim h$ and the desired estimate follows. Here we used the bounds 
\begin{align}\label{eq:interpol-prf-a}
h \|   \nabla \rho_i \|^2_{\Gamma}
+ h^{-1} \|\rho_i \|^2_{\Gamma} 
&\lesssim 
\sum_{m=0}^2 h^{2(m-1)}  \| \rho_i \|^2_{H^m(\Omega_i)}
\\ \label{eq:interpol-prf-b}
\| \rho_i \|^2_{s_h,i} 
&\lesssim 
\sum_{m=1}^2 h^{2(m-1)}  \| \rho_i \|^2_{H^m(O_{h,i})}
\\ \label{eq:interpol-prf-c}
\| \rho_i \|^2_{s_h,\Gamma} 
&\lesssim 
\sum_{m=1}^2 h^{2(m-1)-1}  \| \rho_i \|^2_{H^m(O_{h,\delta,\Gamma})}
\end{align}
To prove (\ref{eq:interpol-prf-a}) we employ the trace inequality 
\begin{equation*}
\| v \|^2_\Gamma \lesssim 
\delta^{-1} \| v \|^2_{U_\delta(\Gamma)\cap \Omega_i} 
+ \delta \| \nabla v \|^2_{U_\delta(\Gamma)\cap \Omega_i}\qquad v \in H^1(\Omega_i)
\end{equation*}
with $\delta \sim h$, to estimate the interface terms involving $\rho_i$ 
as follows
\begin{equation*}
h \| \nabla \rho_i \|^2_{A_i,\Gamma} 
\lesssim 
\| \rho \|^2_{H^1(U_\delta (\Gamma) \cap \Omega_i)}
+
h^2 \| \rho \|^2_{H^2(U_\delta (\Gamma) \cap \Omega_i)}
\lesssim 
\| \rho \|^2_{H^1(O_{h,i})}
+
h^2 \| \rho \|^2_{H^2(O_{h,i})}
\end{equation*}
and 
\begin{equation*}
h^{-1} \| \rho_i \|^2_{\Gamma} 
\lesssim 
h^{-2} \| \rho_i \|^2_{U_\delta (\Gamma) \cap \Omega_i}
+
\| \rho_i \|^2_{H^1(U_\delta (\Gamma) \cap \Omega_i)}
\lesssim  
h^{-2} \| \rho \|^2_{O_{h,i}}
+
\| \rho \|^2_{H^1(O_{h,i})}
\end{equation*}
For (\ref{eq:interpol-prf-b}) we apply the elementwise trace inequality 
\begin{equation*}
\| v \|_F^2 \lesssim h^{-1} \| v \|^2_T + h \| \nabla v \|^2_T
\end{equation*}
which gives
\begin{equation*}
\| \rho \|^2_{s_{h,i}} \lesssim \sum_{m=1}^2 \sum_{T \in T_{h,i}}
(\|\nabla \rho \|^2_T + h \| \nabla^2 \rho \|^2_T) \lesssim \sum_{m=1}^2 h^{2(m-1)} \| v \|^2_{H^m(O_{h,i})} 
\end{equation*}
In a similar way we prove (\ref{eq:interpol-prf-c}), see \cite{BuHaLa15} and 
\cite{LaZa19} for details.
\end{proof}

\subsection{Continuity and Coercivity}
We start with a lemma collecting some useful estimates for expressions involving 
the stabilization parameter $\tau$ and then we prove continuity and coercivity of 
the form $A_h$. 

\begin{lemma} The following estimates related to the stabilization parameter $\tau$ hold
\begin{gather}\label{eq:tau-bounds-a}
 \|B^{-1} \tau B^{-1} + B^{-1}\|_{L^\infty(\Gamma)}  
 \leq \frac{h}{\beta}
 \\ \label{eq:tau-bounds-b}
 \|(B^{-1} \tau - I) \tau^{-1/2} \|_{L^\infty(\Gamma)}   \leq \left(\frac{h}{\beta}\right)^{1/2}
 \\ \label{eq:tau-bounds-c}
 \|\tau\|_{L^\infty(\Gamma)} \leq \frac{\beta}{h}
\end{gather}
\end{lemma}
\begin{proof} First we recall that for any symmetric matrix $D$ it holds
\begin{equation}\label{eq:matrix-norm}
 \| A \|_{\IR^d} \lesssim \max_i |\gamma_i|
\end{equation}
where $\gamma_{i}$ are the eigenvalues of $D$. To prove (\ref{eq:tau-bounds-a}) 
we write $B$ in terms 
of its eigenvalues $\lambda_i$ and eigenvectors $e_i$, 
\begin{equation}
B = \sum_{i=1}^2 \lambda_i e_i \otimes e_i
\end{equation}
and using the definition (\ref{eq:tau-def}) of $\tau$ we obtain the identity 
\begin{align}
 B^{-1} \tau B^{-1} -  B^{-1} 
 = 
 \sum_{i=1}^2 \left( \frac{\tau_i}{\lambda_i} - 1 \right)  \frac{1}{\lambda_i}  e_i \otimes e_i
\end{align}
Here we have the following estimate of the eigenvalues
\begin{equation}
\left| \left( \frac{\tau_i}{\lambda_i} - 1 \right)  \frac{1}{\lambda_i} \right|
=
\left|\left( \frac{\beta}{\lambda_i h + \beta} - 1 \right)  \frac{1}{\lambda_i} \right|
=
\frac{h}{\lambda_i h + \beta} 
\leq 
\frac{h}{\beta}
\end{equation}
which in view of (\ref{eq:matrix-norm})  completes the verification of 
(\ref{eq:tau-bounds-a}).
Next, for (\ref{eq:tau-bounds-b}) we have 
\begin{align}
(B^{-1} \tau - I) \tau^{-1/2} 
=  
\sum_{i=1}^2  \left( \frac{\tau_i}{\lambda_i} - 1 \right)  \frac{1}{\tau_i^{1/2}}  
e_i \otimes e_i
\end{align}
and 
\begin{multline}
\left| \left( \frac{\tau_i}{\lambda_i} - 1 \right)  \frac{1}{\tau_i^{1/2}} \right|
=
\left|\left( \frac{\beta}{\lambda_i h + \beta} - 1 \right)\left( \frac{\lambda_i h + \beta}{\lambda_i \beta} \right)^{1/2} \right|
\\
=
\frac{\lambda_i h}{\lambda_i h + \beta}\left( \frac{\lambda_i h + \beta}{\lambda_i \beta} \right)^{1/2} 
=
\left(\frac{\lambda_i h}{\lambda_i h + \beta}\right)^{1/2}  
\left( \frac{h}{\beta} \right)^{1/2} 
\leq 
\left( \frac{h}{\beta} \right)^{1/2} 
\end{multline}
which proves (\ref{eq:tau-bounds-b}). The final bound (\ref{eq:tau-bounds-c}) 
is a direct consequence of the definition of $\tau$ and the estimate 
\begin{align}
\frac{\lambda_i \beta} {\lambda_i h + \beta}
\leq 
\frac{\lambda_i \beta} {\lambda_i h}
\leq 
\frac{\beta} {h}
\end{align}
\end{proof}

\begin{lemma} There is a constant independent of the eigenvalues of $B$,
  such that for all $v,w\in \tilde V + V_h$,
\begin{align}\label{eq:cont}
\boxed{ \mathcal{A}^R_h (v, w ) \lesssim \tn v \tn_h \tn w \tn_h }
\end{align}
There is a constant independent of the eigenvalues of $B$, such that for all $v \in V_h$, 
\begin{align}\label{eq:coer}
\boxed{
\tn v \tn_h^2 \lesssim  \mathcal{A}^R_h (v, v) 
}
\end{align}
\end{lemma}
\begin{proof} {\bf (\ref{eq:cont}).} Starting from the definition (\ref{eq:Ah}), expanding the terms in $\mathcal{A}^R$, and using Cauchy-Schwarz we obtain
\begin{align}
 \mathcal{A}^R(v,w) &= \sum_{i=1}^2 (A_{i}\nabla v_i , \nabla w_i)_{\Omega_i} 
+(A_\Gamma \nabla_\Gamma v_\Gamma, \nabla_\Gamma w_\Gamma)_\Gamma
\\  \nonumber
&\qquad +
( (n\cdot A \nabla v), (B^{-1} \tau B^{-1} -  B^{-1}) ( n\cdot A \nabla w ) )_\Gamma
\\  \nonumber
&\qquad +
( (n\cdot A \nabla v), (B^{-1} \tau - I) ( w - w_\Gamma) )_\Gamma
\\  \nonumber
&\qquad +
( (n\cdot A \nabla w), (B^{-1} \tau - I) ( v - v_\Gamma) )_\Gamma
\\  \nonumber
&\qquad 
+ (( v - v_\Gamma), \tau ( w - w_\Gamma))_\Gamma
\\ 
&\leq 
\sum_{i=1}^2\|\nabla v_i \|_{A_i,\Omega_i} \| \nabla w_i\|_{A_i,\Omega_i} 
+\| \nabla_\Gamma v_\Gamma\|_{A_\Gamma,\Gamma} 
\| \nabla_\Gamma w_\Gamma \|_{A_i,\Gamma}
\\   \nonumber
&\qquad +
\|n\cdot A \nabla v\|_\Gamma \|B^{-1} \tau B^{-1} - B^{-1}\|_{L^\infty(\Gamma)} 
\| n\cdot A \nabla w \|_{\Gamma}
\\  \nonumber
&\qquad +
\|n\cdot A \nabla v\|_\Gamma \|( B^{-1} \tau - I) \tau^{-1/2}\|_{L^\infty(\Gamma)} \|w - w_\Gamma\|_{\tau,\Gamma}
\\  \nonumber
&\qquad +
\|n\cdot A \nabla w\|_\Gamma \|(B^{-1} \tau - I)\tau^{-1/2} \|_{L^\infty(\Gamma)} 
\|v - v_\Gamma\|_{\tau,\Gamma}
\\  \nonumber
&\qquad 
+ \|  v - v_\Gamma\|_{\tau,\Gamma} \|w - w\|_{\tau,\Gamma}
\\
&=\bigstar
\end{align}
Using the estimates (\ref{eq:tau-bounds-a})-(\ref{eq:tau-bounds-b})  we obtain
\begin{align}
\bigstar 
&\leq
\sum_{i=1}^2 \|\nabla v_i\|_{A_i,\Omega_i} \|\nabla w_i\|_{A_i,\Omega_i} 
+\| \nabla_\Gamma v_\Gamma\|_{A_\Gamma,\Gamma} \|\nabla_\Gamma w_\Gamma\|_{A_\Gamma,\Gamma}
\\ \nonumber
&\qquad +
\beta^{-1} h \|n\cdot A \nabla v\|_\Gamma 
\| n\cdot A \nabla w \|_{\Gamma}
\\  \nonumber
&\qquad +
\beta^{-1/2} h^{1/2} \|n\cdot A \nabla v\|_\Gamma  \|w - w_\Gamma\|_{\tau,\Gamma}
\\  \nonumber
&\qquad +
\beta^{-1/2} h^{1/2}  \|n\cdot A \nabla w\|_\Gamma 
   \|v - v_\Gamma\|_{\tau,\Gamma}
\\  \nonumber
&\qquad 
+\|  v - v_\Gamma\|_{\tau,\Gamma} \|w - w\|_{\tau,\Gamma}
\\ 
&\leq
\sum_{i=1}^2 \|\nabla v_i\|_{A_i,\Omega_i} \|\nabla w_i\|_{A_i,\Omega_i} 
+\| \nabla_\Gamma v_\Gamma\|_{A_\Gamma,\Gamma} \|\nabla_\Gamma w_\Gamma\|_{A_\Gamma,\Gamma}
\\  \nonumber
&\qquad +
(\beta^{-1} \|n\|^2_{A,\infty,\Gamma}) h^{1/2} \|\nabla v\|_{A,\Gamma} 
h^{1/2} \| \nabla w \|_{\Gamma}
\\  \nonumber
&\qquad +
(\beta^{-1} \|n\|_{A,\infty,\Gamma} )^{1/2} h^{1/2} \|\nabla v\|_{A,\Gamma}  \|w - w_\Gamma\|_{\tau,\Gamma}
\\  \nonumber
&\qquad +
(\beta^{-1}  \|n\|_{A,\infty,\Gamma} )^{1/2} h^{1/2}  \|\nabla w\|_{A,\Gamma} 
   \|v - v_\Gamma\|_{\tau,\Gamma}
\\  \nonumber
&\qquad 
+\|  v - v_\Gamma\|_{\tau,\Gamma} \|w - w\|_{\tau,\Gamma}
\\  
&\lesssim \tn v \tn_h \tn w \tn_h
\end{align}
where we used the bound $\beta^{-1}  \|n\|_{A,\infty,\Gamma} \lesssim
1$, see (\ref{eq:tau-beta-condition}). By the Cauchy-Schwarz
inequality we have $s_h(v,w) \lesssim \tn v \tn_h \tn w \tn_h$.

\paragraph{\bf (\ref{eq:coer}).} To prove the coercivity we have the identity 
\begin{align}
\mathcal{A}^R_h(v,v) 
&= \sum_{i=1}^2 (A_{i}\nabla v_i , \nabla v_i)_{\Omega_i} 
+(A_\Gamma \nabla_\Gamma v_\Gamma, \nabla_\Gamma v_\Gamma)_\Gamma
+ s_h(v,v)
\\  \nonumber
&\qquad +
( (n\cdot A \nabla v), (B^{-1} \tau B^{-1} -  B^{-1}) ( n\cdot A \nabla v ) )_\Gamma
\\  \nonumber
&\qquad +
2( (n\cdot A \nabla v), (B^{-1} \tau - I) ( v - v_\Gamma) )_\Gamma
\\  \nonumber 
&\qquad 
+ (( v - v_\Gamma), \tau ( v - v_\Gamma))_\Gamma
\\
&\geq
\sum_{i=1}^2 \|\nabla v_i\|_{A_i,\Omega_i}^2
+\| \nabla_\Gamma v_\Gamma\|_{A_\Gamma,\Gamma}^2 + \|v\|^2_{s_h}
\\  \nonumber
&\qquad - \beta^{-1} \|n\|^2_{A,\infty,\Gamma} h \|\nabla v\|^2_{A,\Gamma}
\\  \nonumber
&\qquad -
 2 (\beta^{-1}  \|n\|_{A,\infty,\Gamma})^{1/2} h^{1/2} \|\nabla v\|_{A,\Gamma}  
 \|v - v_\Gamma\|_{\tau,\Gamma}
\\  \nonumber
&\qquad 
+ \|  v - v_\Gamma\|^2_{\tau,\Gamma}
\end{align}
We conclude the argument as usual by estimating the negative terms as 
follows
\begin{align}
& \beta^{-1} \|n\|^2_{A,\Gamma} h \|\nabla v\|^2_{A,\Gamma}
+ 2 (\beta^{-1}  \|n\|_{A,\infty,\Gamma})^{1/2} h^{1/2} \|\nabla v\|_{A,\Gamma}  \|v - v_\Gamma\|_{\tau,\Gamma}
\\
&\qquad \leq  
 3 \beta^{-1} \|n\|^2_{A,\infty,\Gamma} h \|\nabla v\|^2_{A,\Gamma}
 +  \frac{1}{2} \|v - v_\Gamma\|^2_{\tau,\Gamma}
 \\
 &\qquad \leq 3 \beta^{-1} \|n\|^2_{A,\infty,\Gamma} 
 C_I\left( \sum_{i=1}^2 \| \nabla v_i\|^2_{A_i, \Omega_i}  + \| v \|^2_{s_{h,i}} \right)
 +  \frac{1}{2} \|v - v_\Gamma\|^2_{\tau,\Gamma}
 \\
 &\qquad 
 \leq
   \frac{1}{2}  \left( \sum_{i=1}^2 \| \nabla v_i\|^2_{A_i, \Omega_i} + \| v \|^2_{s_{h,i}} \right)
 +  \frac{1}{2} \|v - v_\Gamma\|^2_{\tau,\Gamma}
\end{align}
Here we used the inverse estimate 
\begin{equation}
h \|\nabla v_i \|^2_{A_i,\Gamma}  
\leq
C_I( \| \nabla v_i\|^2_{A_i, \Omega_i}   + \| v \|^2_{s_{h,i}} )
\end{equation}
which follows from the inverse bound 
\begin{equation}
h \|\nabla v_i \|^2_{A_i,\Gamma \cap T} 
\lesssim
 h \|\nabla v_i \|^2_{\Gamma \cap T} 
\lesssim 
\|\nabla v_i \|^2_{T} 
\lesssim 
\|\nabla v_i \|^2_{A_i,T} 
\end{equation}
 together 
with (\ref{eq:stab-bulk}), and finally, we chose $\beta$ large enough to guarantee 
that 
\begin{equation}
3 \beta^{-1} \|n\|^2_{A,\infty,\Gamma} C_I \leq \frac{1}{2}
\end{equation}
We conclude that 
\begin{equation}
\mathcal{A}^R_h(v,v) \geq \frac{1}{2} \tn v \tn_h^2
\end{equation}
which completes the proof.
\end{proof}

\subsection{A priori Error Estimates}
In this section we prove error estimates for the approximate solution $u_h$.
\begin{theorem} \label{thm:tnorm_error}
Let $u \in \tilde V$ be the solution of \eqref{eq:prob-weak} and $u_h
\in V_h$ be the
solution of \eqref{eq:robustFEM}. Then there is a constant not dependent on the matrix $B$ in the interface condition (\ref{eq:interface}) such that 
\begin{equation}
\boxed{\tn u - u_h \tn_h \lesssim  h \left( \sum_{i=1}^2 \| f_i \|_{L^2(\Omega_i)} + \| f_\Gamma \|_{L^2(\Gamma)} \right)}
\end{equation}
\end{theorem}
\begin{proof} 
First we decompose the error in the approximation error and the
discrete error $u-u_h =  u - \pi_h u +  \pi_h u - u_h$ and note that
by the triangle inequality 
\begin{equation}
\tn u - u_h \tn_h \lesssim \tn u - \pi_h u \tn_h+ \tn \pi_h u - u_h \tn_h.
\end{equation}
The first term on the right hand side is bounded by \eqref{eq:interpol}.
For the second term on the right hand side, using coercivity \eqref{eq:coer},
Galerkin orthogonality \eqref{eq:galortho}, and continuity
\eqref{eq:cont} we obtain 
\begin{align}
\tn \pi_h u - u_h \tn_h^2 &\lesssim \mathcal{A}^R_h(\pi_h u - u_h, \pi_h u - u_h ) 
\\
&= \mathcal{A}^R(\pi_h u-u, \pi_h u - u_h ) + s_h(\pi_h u, \pi_h u - u_h )
\\
&\lesssim \tn \pi_h u-u \tn_h \tn \pi_h u - u_h  \tn_h.
\end{align}
In the last inequality we used that if $u^e := (E u_1, E u_2, E_\Gamma
u_\Gamma) \in \tilde V$ then 
\begin{equation}
s_h(\pi_h u, \pi_h u - u_h )=s_h(\pi_h u - u^e, \pi_h u - u_h )
\lesssim \tn \pi_h u-u \tn_h \tn \pi_h u - u_h  \tn_h.
\end{equation}
Thus 
\begin{equation}
\tn u -u_h \tn_h 
\lesssim 
\tn u - \pi_h u \tn_h 
\lesssim 
h \left( \sum_{i=1}^2 \| u_i \|_{H^2(\Omega_i)} + \| u_\Gamma \|_{H^2(\Gamma)} \right)
\end{equation}
where we used the interpolation error estimate (\ref{eq:interpol}). To
conclude we apply the regularity estimate \eqref{eq:ell-reg}.
\end{proof}
\begin{corollary}
Under the same assumptions as for Theorem \ref{thm:tnorm_error} there holds
\begin{equation}
s_h(u_h,u_h) \lesssim h \left( \sum_{i=1}^2 \| u_i \|_{H^2(\Omega_i)} + \| u_\Gamma \|_{H^2(\Gamma)} \right).
\end{equation}
\end{corollary}
\begin{proof}
Using the triangle inequality we see that
\begin{equation}
\|u_h\|_{s_h} \leq \|\pi_h u\|_{s_h}+\|\pi_h u - u_h\|_{s_h}
\end{equation}
The second term on the right hand side is bounded by the arguments of
Theorem \ref{thm:tnorm_error}. For the first term on the right hand
side recall that by the consistency properties of $s_h$ and the
construction of $\pi_h u$ there holds
\begin{equation}
s_h(\pi_h u,\pi_h u) =s_h(u^e-\pi_h u, u^e - \pi_h u)
\end{equation}
We conclude the proof by applying \eqref{eq:interpol-prf-b}-\eqref{eq:interpol-prf-c}.
\end{proof}
The following error estimate in the $L^2$-norm also holds
\begin{theorem}
Let $u \in \tilde V$ be the solution of \eqref{eq:prob-weak} and $u_h \in
V_h$ be the
solution of \eqref{eq:robustFEM}. Then there holds
\begin{equation}
\boxed{
\|u_h - u\|_{\Omega} + \|u_{h,\Gamma} - u_\Gamma\|_\Gamma \lesssim h^2
\left(\sum_{i=1}^2 \|f_i\|_{L^2(\Omega_i)} +
  \|f_\Gamma\|_{L^2(\Gamma)}\right)
}
\end{equation}
\end{theorem} 
\begin{proof}
For $\psi_\Omega \in L^2(\Omega)$ and $\psi_\Gamma$, such that $\| \psi_\Omega \|_{\Omega} + \| \psi_\Gamma \|_\Gamma=1$ let
$\varphi:=(\varphi_1,\varphi_2, \varphi_\Gamma) \in V$ be the weak solution to 
\begin{equation}
\mathcal{A}(v,\varphi) = (\psi_\Omega,v)_\Omega +  (\psi_\Gamma,v_\Gamma)_\Gamma.
\end{equation}
Then by \eqref{eq:ell-reg} we have 
\begin{equation}\label{eq:dual-reg}
\| \varphi_1 \|_{H^2(\Omega_1)} + \| \varphi_2 \|_{H^2(\Omega_2)}  
+ \| \varphi_\Gamma \|_{H^2(\Gamma)} 
\lesssim    
\| \psi_\Omega \|_{\Omega} + \| \psi_\Gamma \|_\Gamma\sim 1.
\end{equation}
Let $e = (u_1-u_{1,h},u_2-u_{2,h},u_\Gamma-u_{\Gamma,h})$ and observe
that using \eqref{eq:robust_consist},
\begin{equation}
(\psi_\Omega,u_1-u_{1,h})_{\Omega_1}
+(\psi_\Omega,u_2-u_{2,h})_{\Omega_2}  +
(\psi_\Gamma,u_\Gamma -u_{h,\Gamma})_\Gamma = \mathcal{A}(e,\varphi) = \mathcal{A}^R(e,\varphi).
\end{equation}
Applying now the Galerkin orthogonality \eqref{eq:galortho} we see that
\begin{equation}
\mathcal{A}^R(e,\varphi) = \mathcal{A}^R(e,\varphi-\pi_h \varphi) -
s_h(e, \pi_h \varphi) = \mathcal{A}^R_h(e,\varphi-\pi_h \varphi).
\end{equation}
By the continuity \eqref{eq:cont} we can bound the right hand side,
\begin{equation}
\mathcal{A}^R_h(e,\varphi-\pi_h \varphi) \lesssim \tn e \tn_h \tn
\varphi-\pi_h \varphi \tn_h.
\end{equation}
Then applying the approximation \eqref{eq:interpol} and the regularity
\eqref{eq:dual-reg} we have
\begin{equation}
(\psi_\Omega,u_1-u_{1,h})_{\Omega_1}
+(\psi_\Omega,u_2-u_{2,h})_{\Omega_2}  +
(\psi_\Gamma,u_\Gamma -u_{h,\Gamma})_\Gamma \lesssim h \tn e \tn_h.
\end{equation}
We conclude by applying Theorem \ref{thm:tnorm_error} in the right hand side and
taking the supremum over the functions $(\psi_\Omega,\psi_\Gamma)$ in
$L^2$.
\end{proof}

\section{Numerical Examples} 

In this Section we illustrate the properties of the model and method by presenting 
some numerical  results. In all examples we used $\beta=10$ as a stabilization parameter.

\subsection{Convergence and Robustness with Respect to Conditioning}\label{ex_1}

We consider a simple example with known exact solution: the domain $(0,1)\times (0,1)$ is cut 
in half along a vertical line at $x=1/2$.
We take 
$A_1 = A_2 = A_{\Gamma} = I$
and choose a problem with exact solution $u=x(1-x)y(1-y)$. This solution corresponds (without coupling) to 
the source terms $$f_i=2x(1-x)+2y(1-y)$$ Since the normal derivative of the exact solution is zero at $x=1/2$, it does not contribute to the source term on the interface.
We choose $f_{\Gamma} = 1/2$ corresponding to $u_{\Gamma} =y(1-y)/4$, and thus $u_{\Gamma}=u$ at $x=1/2$. We apply zero Dirichlet boundary conditions on $u$
and on $u_{\Gamma}$ (imposed on the boundary of the band of elements intersected by $(1/2,y)$). This is now the solution of (\ref{eq:lapbelt1})--(\ref{eq:lapbelt2}) independent of $B$. A sample discrete solution is shown in Fig. \ref{fig:elevation} with $u_{\Gamma}$ shown as a red line.
We did not impose gradient jumps on the band (second term in $s_{h,\Gamma}$), normal stabilization proved sufficient in this case.

In Figs. \ref{fig:conv1}--\ref{fig:conv4} we show convergence for different choices of parameters in different norms. The method is completely robust with
optimal convergence for all choices. In Fig. \ref{fig:condest} we show the variation of the condition number ((left) with respect to mesh refinement and choice of $\alpha$. The condition number is $O(h^{-2})$ as expected and does not grow with $\alpha$. We also show (right) the effect of using the non--robust method (\ref{eq:FEM}) which shows a linear dependece on $\alpha$ on a fixed mesh, while no such effect is present in the robust method. This robustness
is important since $\alpha$ physically depends on the crack width \cite{MaJaRo05} which is expected to be small.

\subsection{Effect of Gradient Jump Stabilization}\label{ex_2}

This example is taken from \cite{MaJaRo05} with domain is $(0,2)\times(0,1)$ with Dirichlet data $u=1$ at $x=2$ and $u=0$ at $x=0$. Homogeneous Neumann data were applied at $y=0$ and $y=1$. Data were $f_i=f_\Gamma=0$, $A_1 = A_2= I$ and $A_{\Gamma} = a_{\Gamma}d\, I$
with $a_{\Gamma}=2\times 10^{-3}$ for $1/4 < y< 3/4$, $a_{\Gamma}=1$ elsewhere, and with $d=0.01$ (the thickness of the crack). Following \cite{MaJaRo05} we then set $\alpha= 2a_{\Gamma}/d$. 

To show the effect of stabilization, we chose to scale $s_{h,i}$ and $s_{h,\Gamma}$ by a parameter $\gamma$. We retained $\beta = 10$ and normal stabilization on the band. In Figs. \ref{fig:elevationzero}--\ref{fig:elevationone} we show the effect of the parameter $\gamma$. When $\gamma = 0$ the jump in diffusion on the interface leads to slight instabilities at $y=1/4$ and $y=3/4$ which are visible to the eye. These are less pronounced for $\gamma =10^{-2}$ and not significant for $\gamma=1$. The overall solution agrees with that of \cite{MaJaRo05}.

\subsection{Physical Effect of Crack Width}\label{ex_3}

Finally, we show the effect of the crack width with respect to the solution. We used a domain $(0,1)\times(0,1)$ with a quarter circle crack, shown on the computational mesh in Fig. \ref{fig:mesh}. The data were $A_1 = 5\, I$ (inside the circle) $A_2= I$ (outside the circle) and $a_{\Gamma}=0.1$ with definitions as in Example \ref{ex_2}. Dirichlet data $u=1$ at $x=1$ and $u=0$ at $x=0$ were used (also on the band) and homogeneous Neumann data on the remaining boundaries. In Figs. \ref{fig:dminus2}--\ref{fig:dminus4} we see the effect of decreasing the interface width by one order of magnitude between figures. The solution rapidly tends to a continuous state.

\bigskip

\paragraph{Acknowledgement.} This research was supported in part by
EPSRC, UK, Grant No. EP/P01576X/1, the Swedish Foundation for Strategic Research Grant No.\ AM13-0029, the Swedish Research Council Grants No.
2013-4708,  2017-03911, 2018-05262, and Swedish strategic research programme eSSENCE.

\bibliographystyle{abbrv}

%
%
%

\begin{figure}[ht]
	\begin{center}
		\includegraphics[scale=0.25]{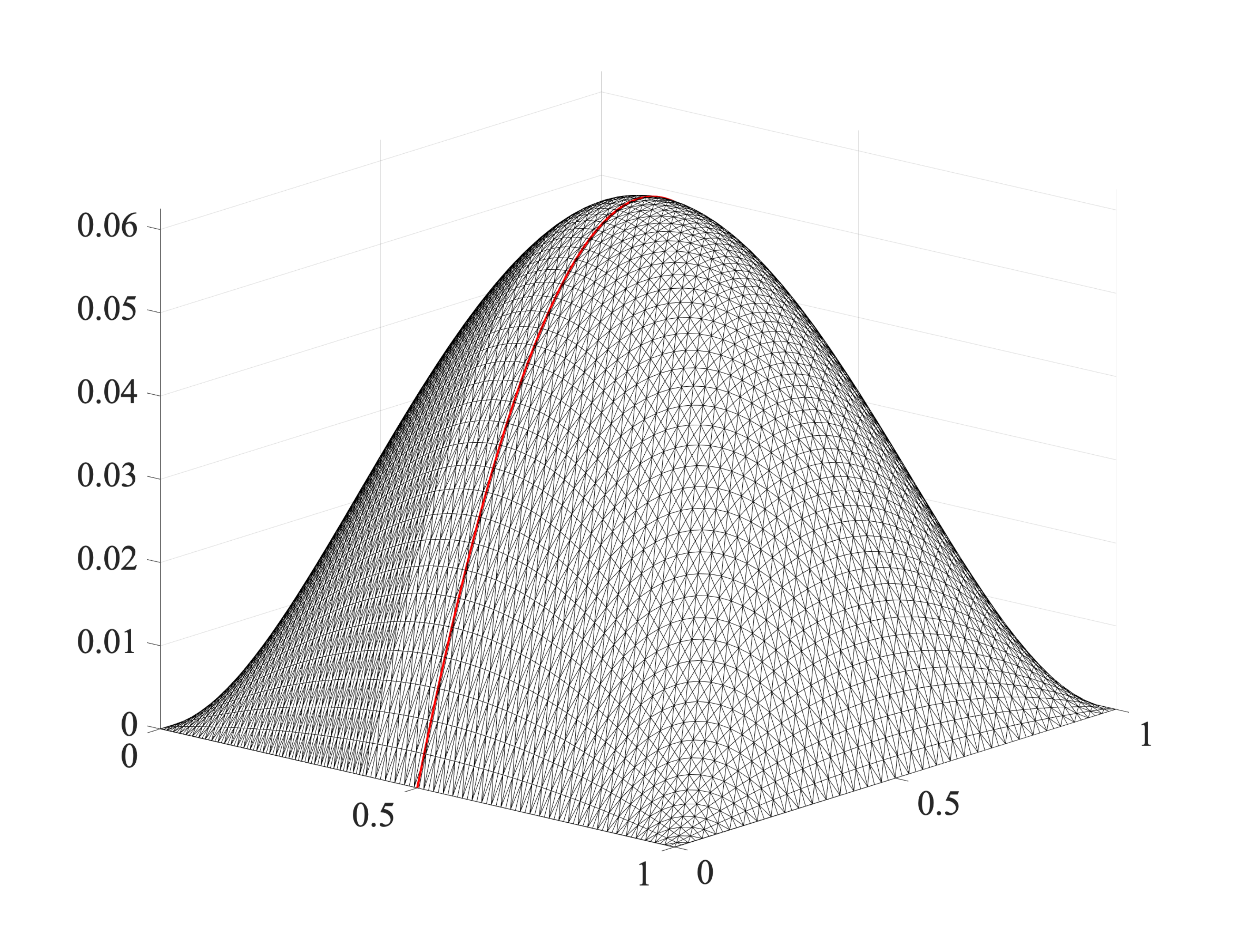}
	\end{center}
	\caption{Elevation of the computed solution on a particular mesh (for $\alpha =1$, $\xi=1$).}
	\label{fig:elevation}
\end{figure}

\begin{figure}[ht]
	\begin{center}
		\includegraphics[scale=0.15]{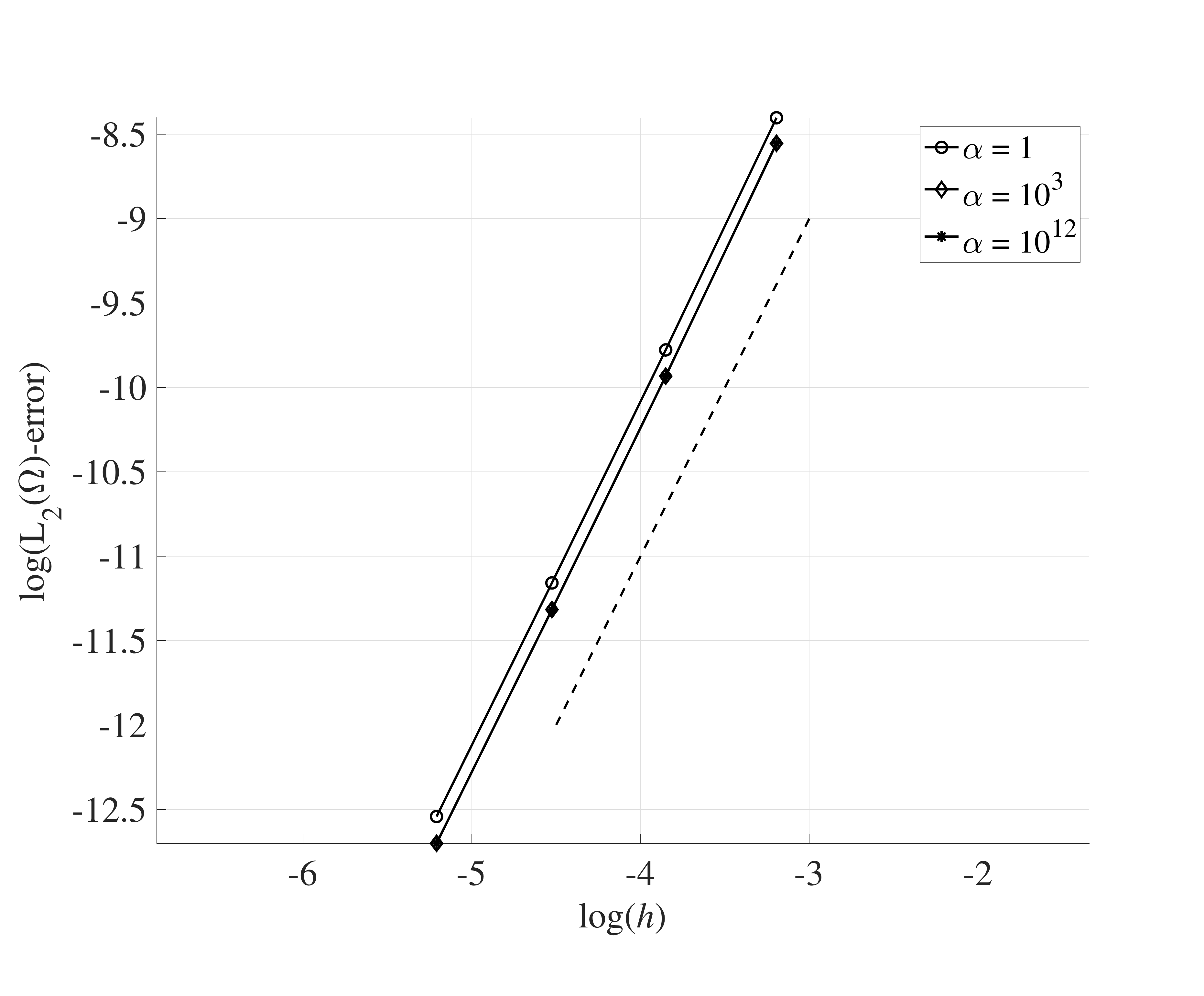}\includegraphics[scale=0.14]{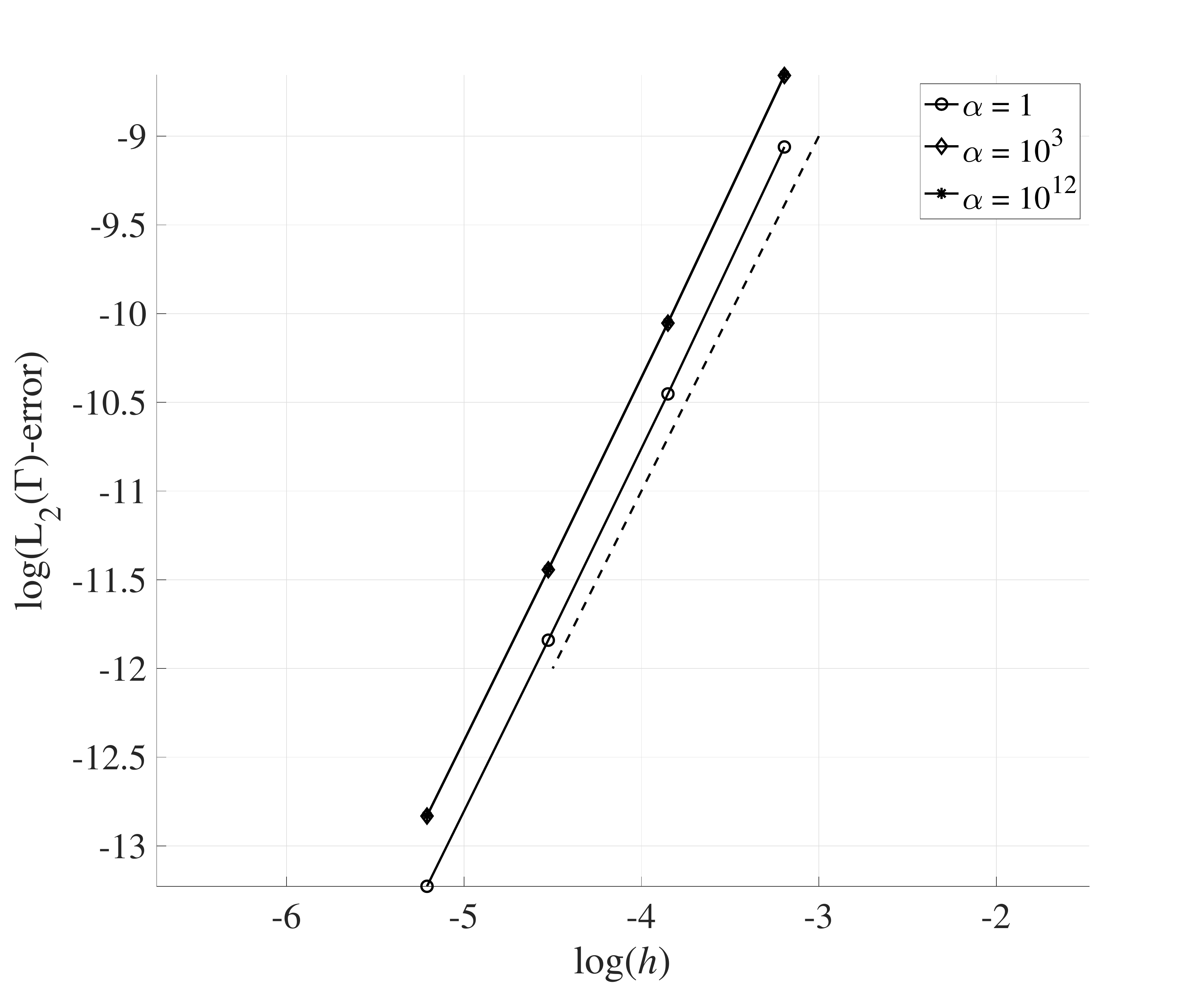}
	\end{center}
	\caption{Convergence in $L_2(\Omega)$ and in $L_2(\Gamma)$ for varying $\alpha$ with $\xi=1$. Dashed line has inclination 1:2.}
	\label{fig:conv1}
\end{figure}

\begin{figure}[ht]
	\begin{center}
		\includegraphics[scale=0.15]{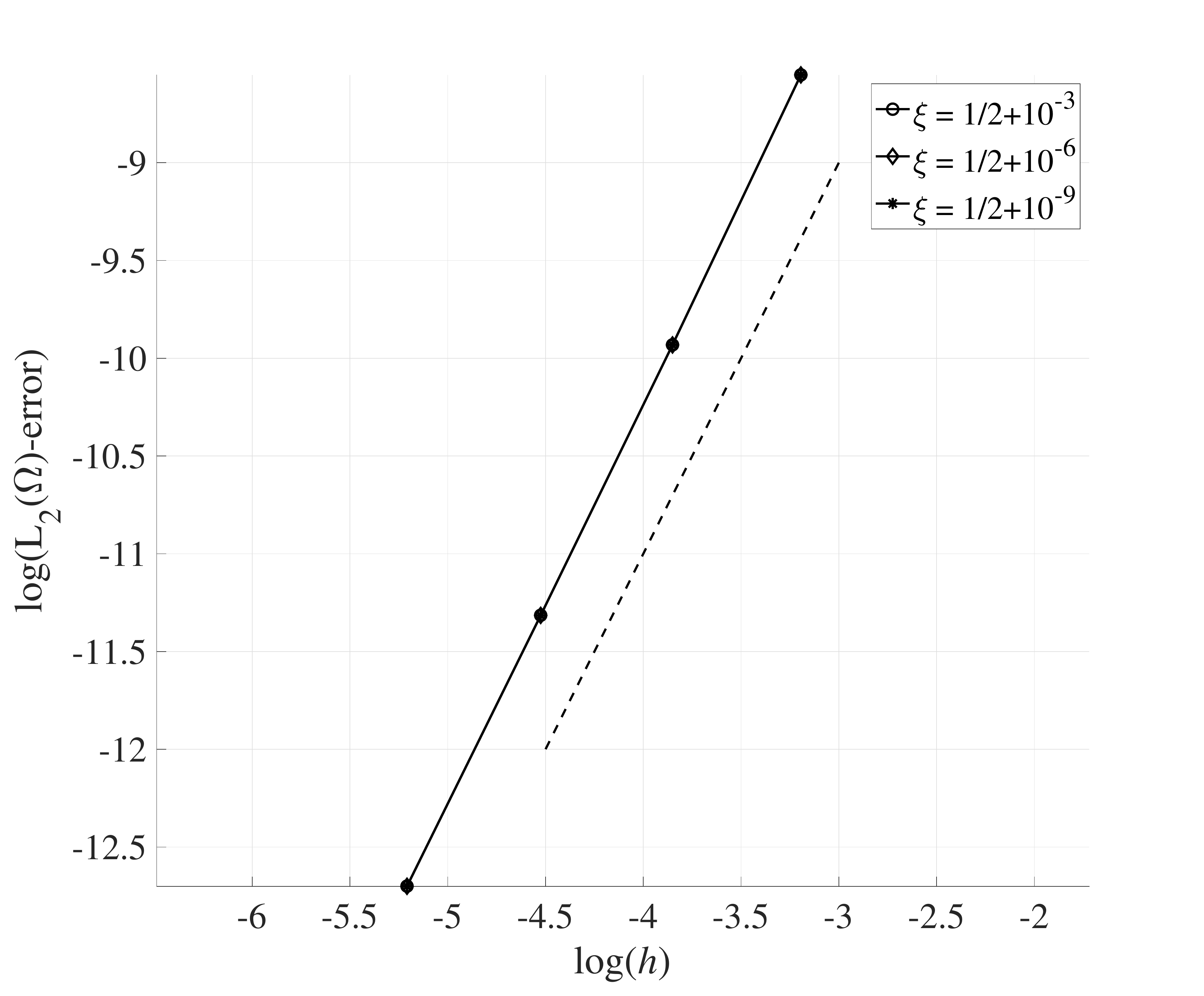}\includegraphics[scale=0.15]{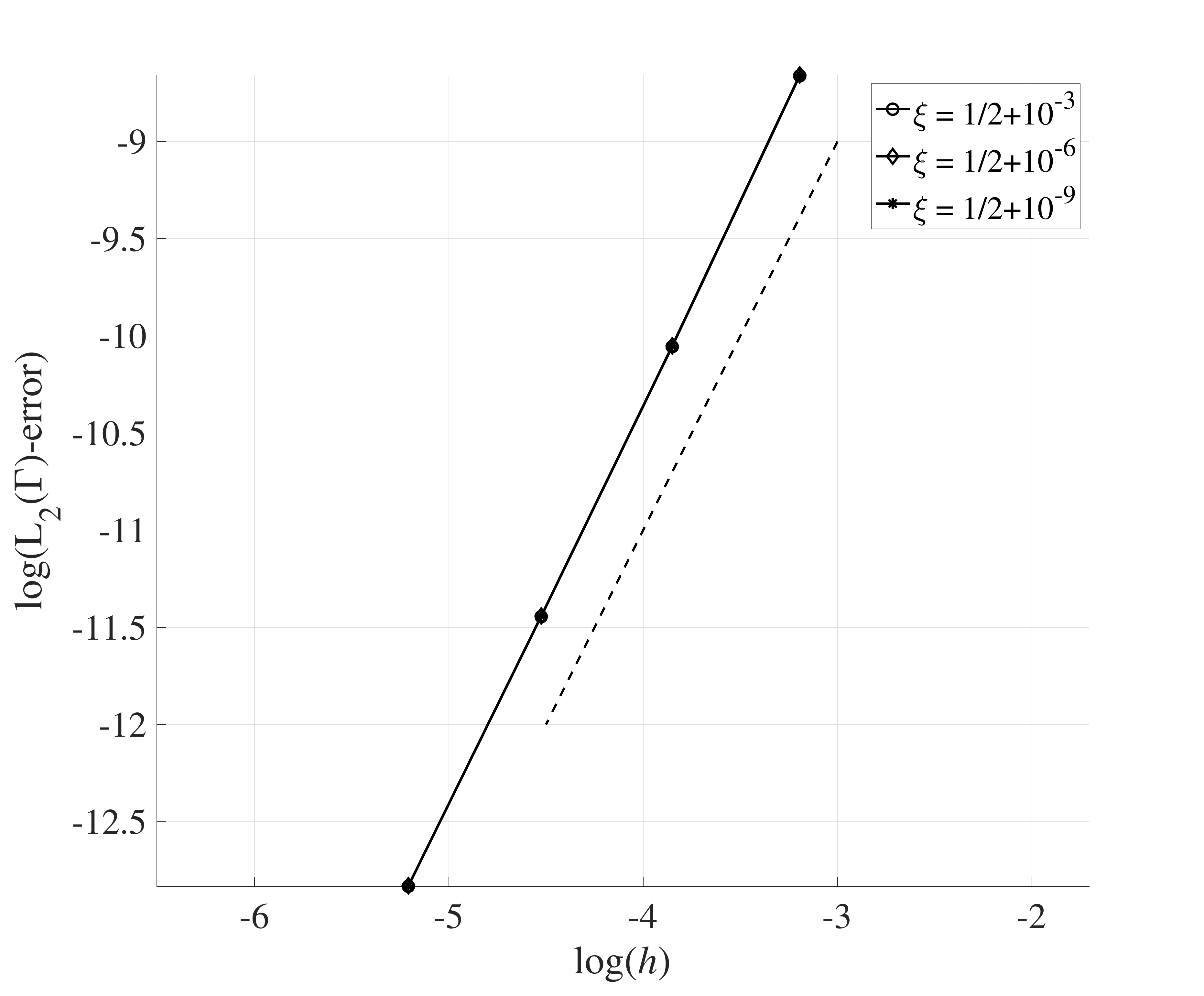}
	\end{center}
	\caption{Convergence in $L_2(\Omega)$ and in $L_2(\Gamma)$ for varying $\xi$ with $\alpha=1$. Dashed line has inclination 1:2.}
	\label{fig:conv2}
\end{figure}

\begin{figure}[ht]
	\begin{center}
		\includegraphics[scale=0.15]{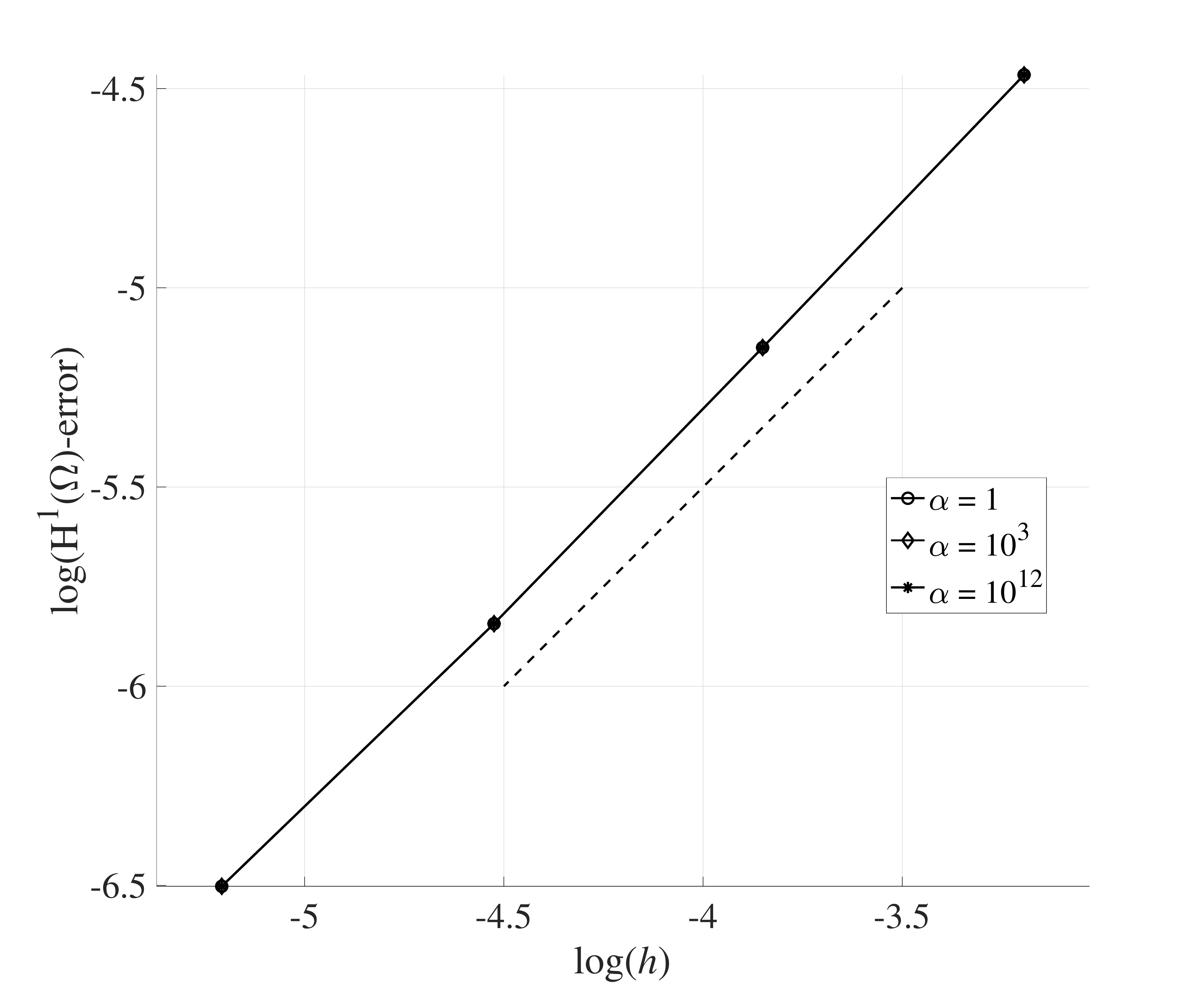}\includegraphics[scale=0.15]{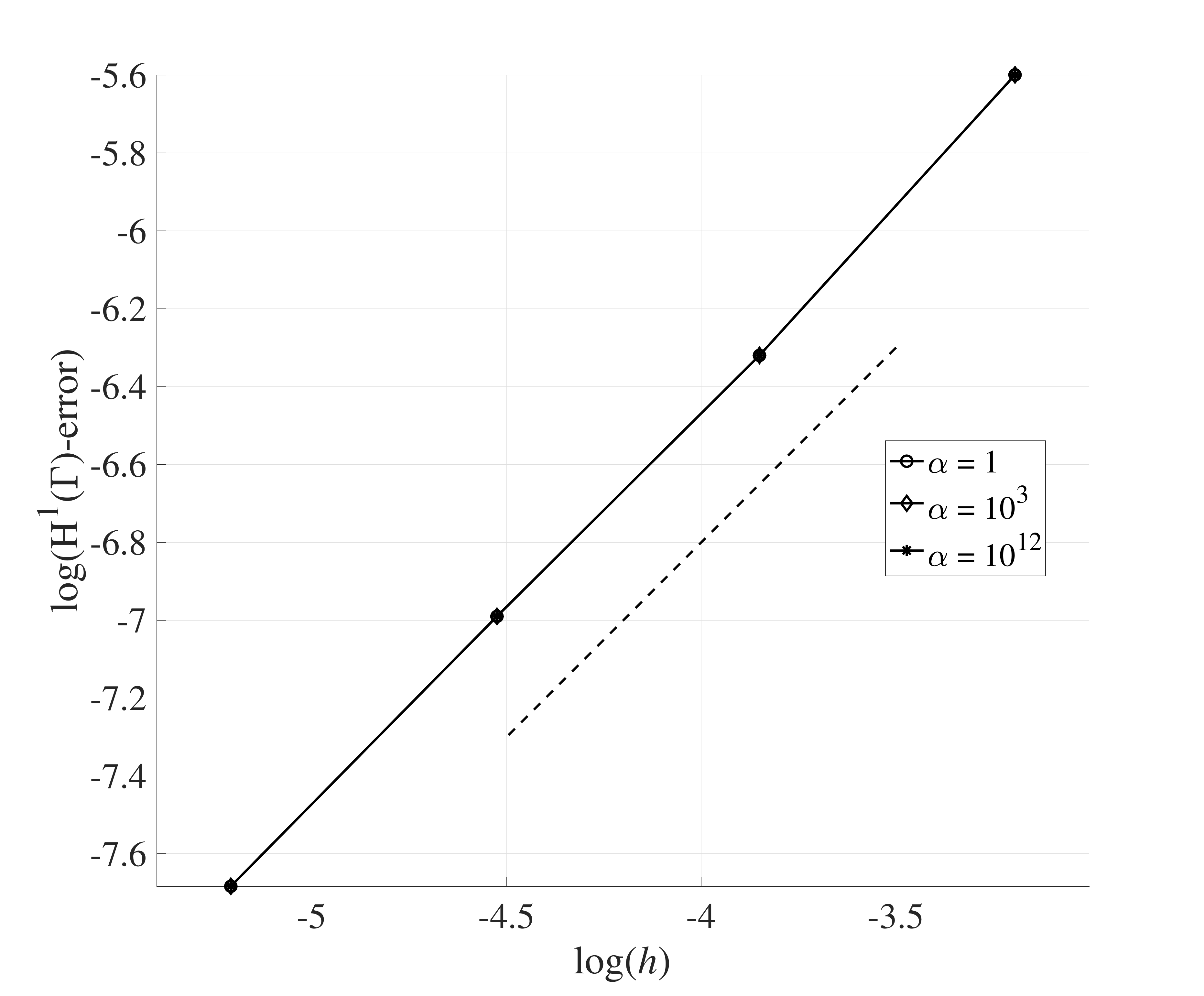}
	\end{center}
	\caption{Convergence in $H^1(\Omega)$ and in $H^1(\Gamma)$ for varying $\alpha$ with $\xi=1$. Dashed line has inclination 1:1.}
	\label{fig:conv3}
\end{figure}

\begin{figure}[ht]
	\begin{center}
		\includegraphics[scale=0.16]{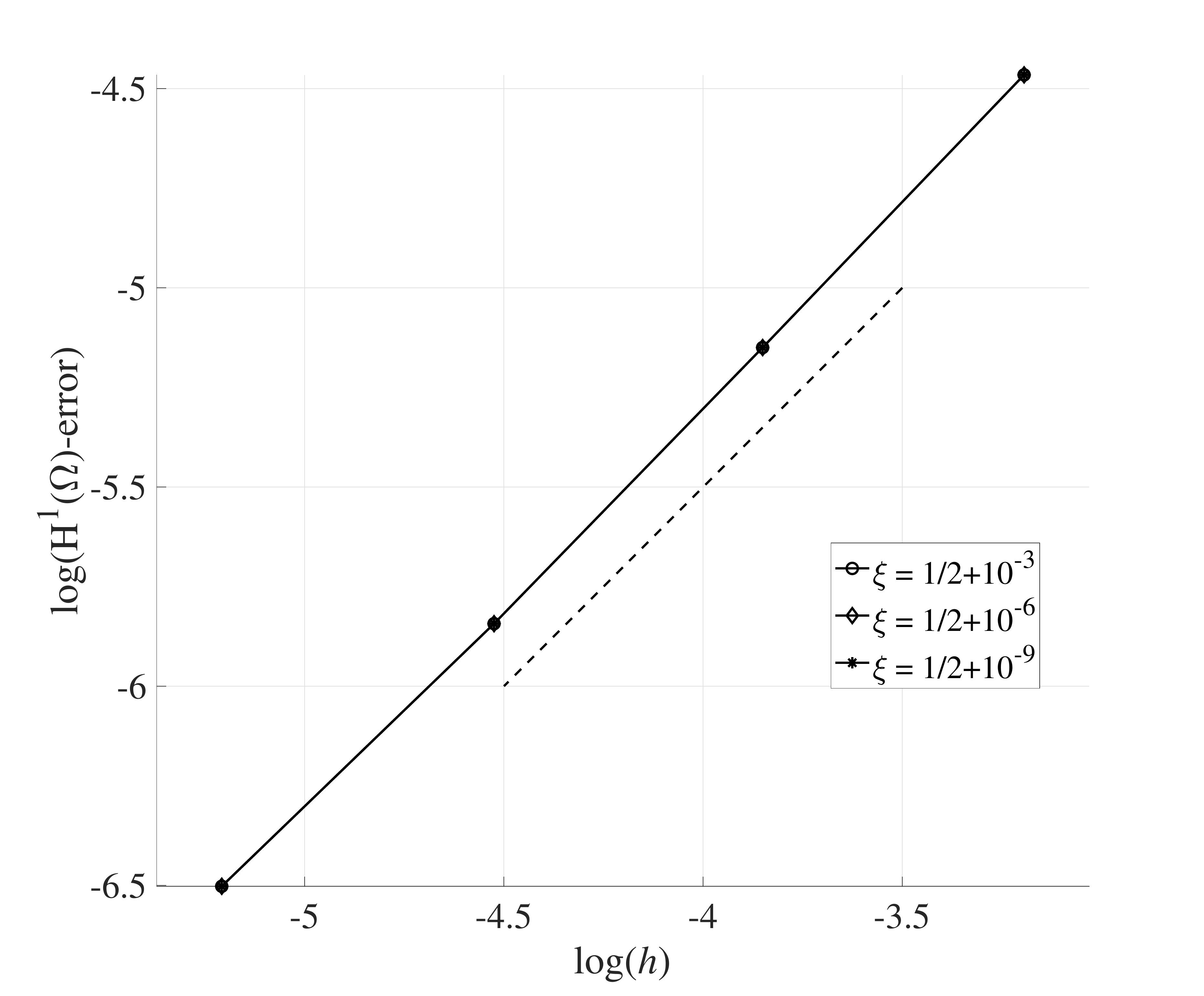}\includegraphics[scale=0.16]{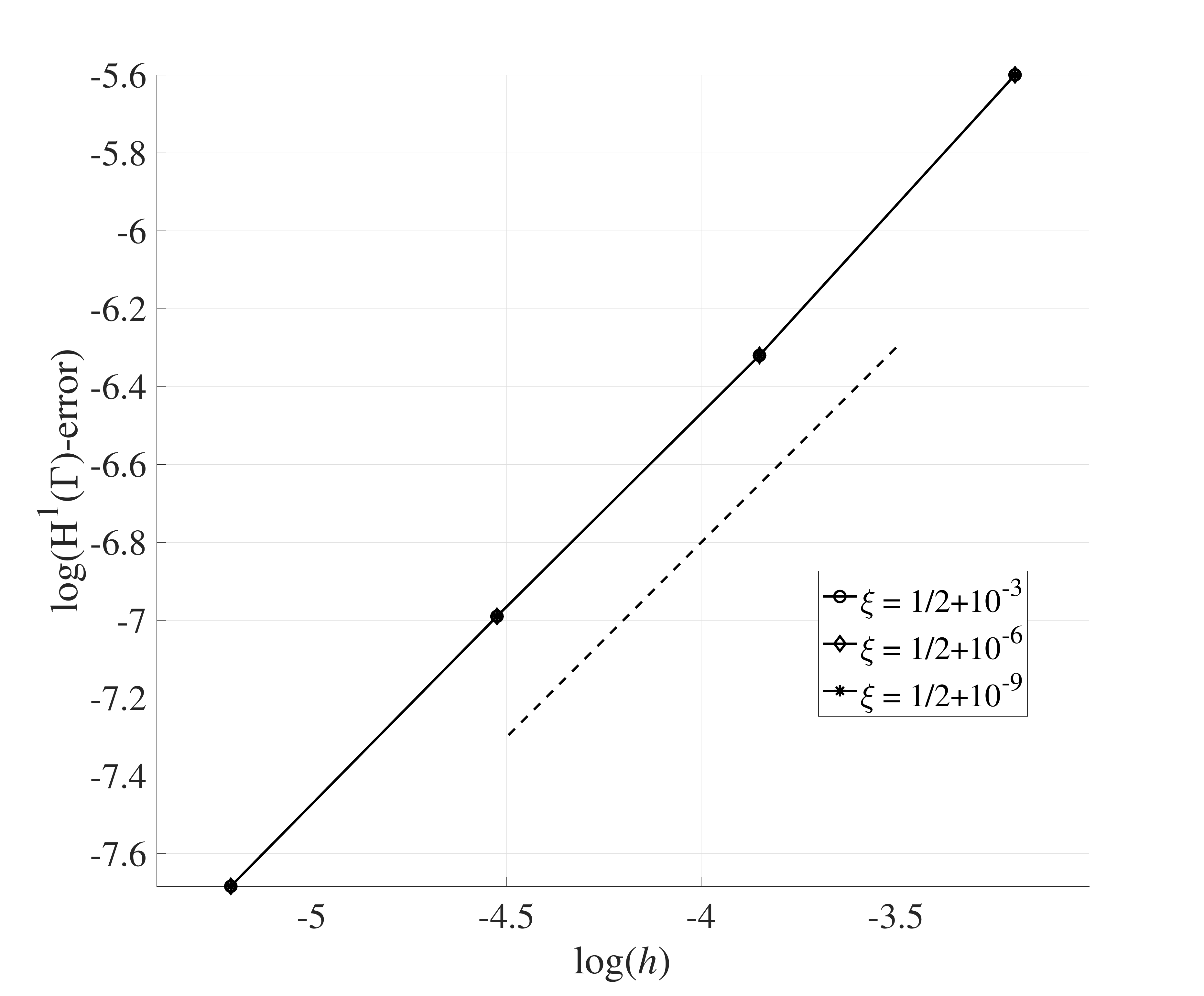}
	\end{center}
	\caption{Convergence in $H^1(\Omega)$ and in $H^1(\Gamma)$ for varying $\xi$ with $\alpha=1$. Dashed line has inclination 1:1.}
	\label{fig:conv4}
\end{figure}

\begin{figure}[ht]
	\begin{center}
		\includegraphics[scale=0.16]{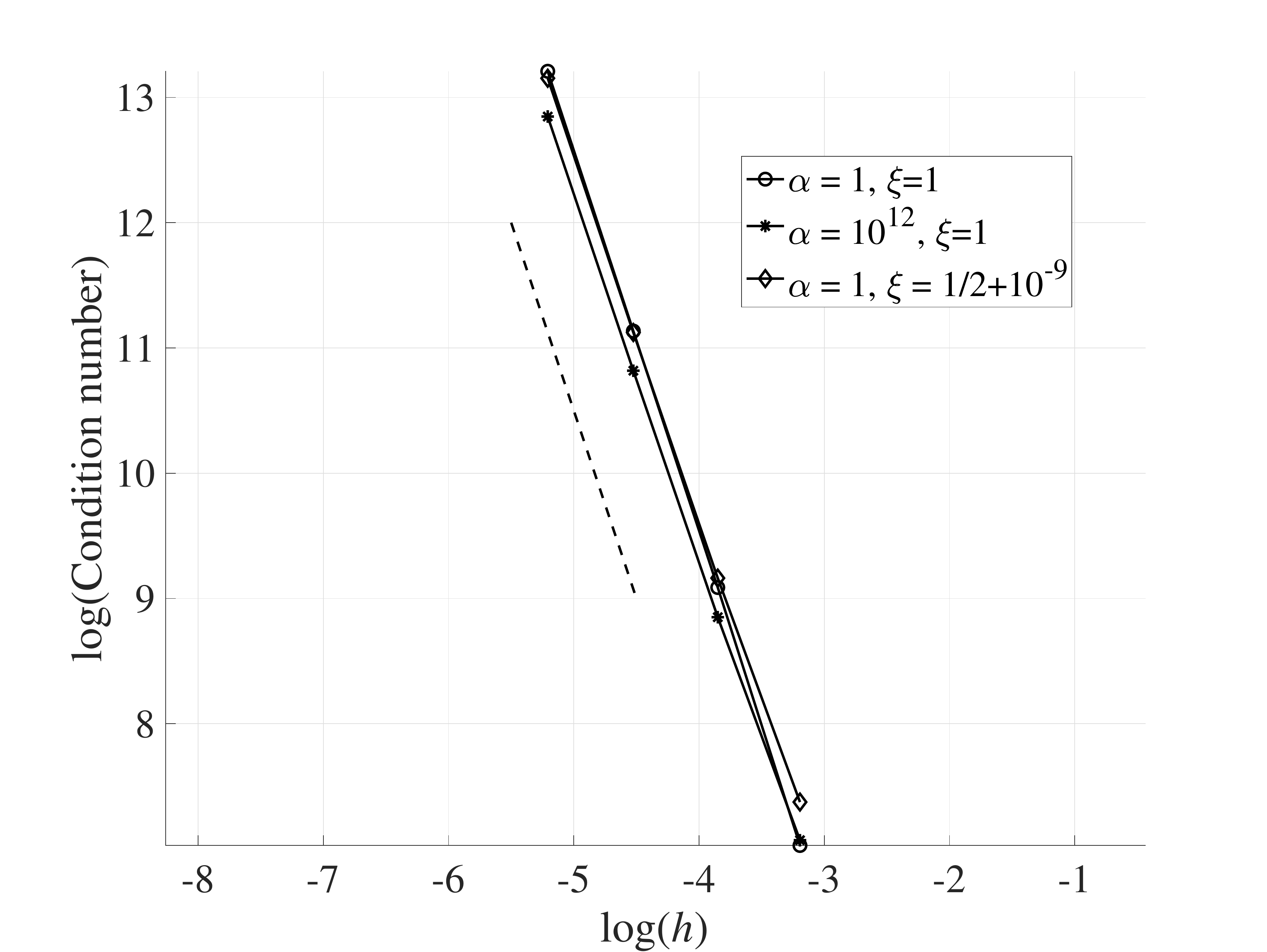}\includegraphics[scale=0.16]{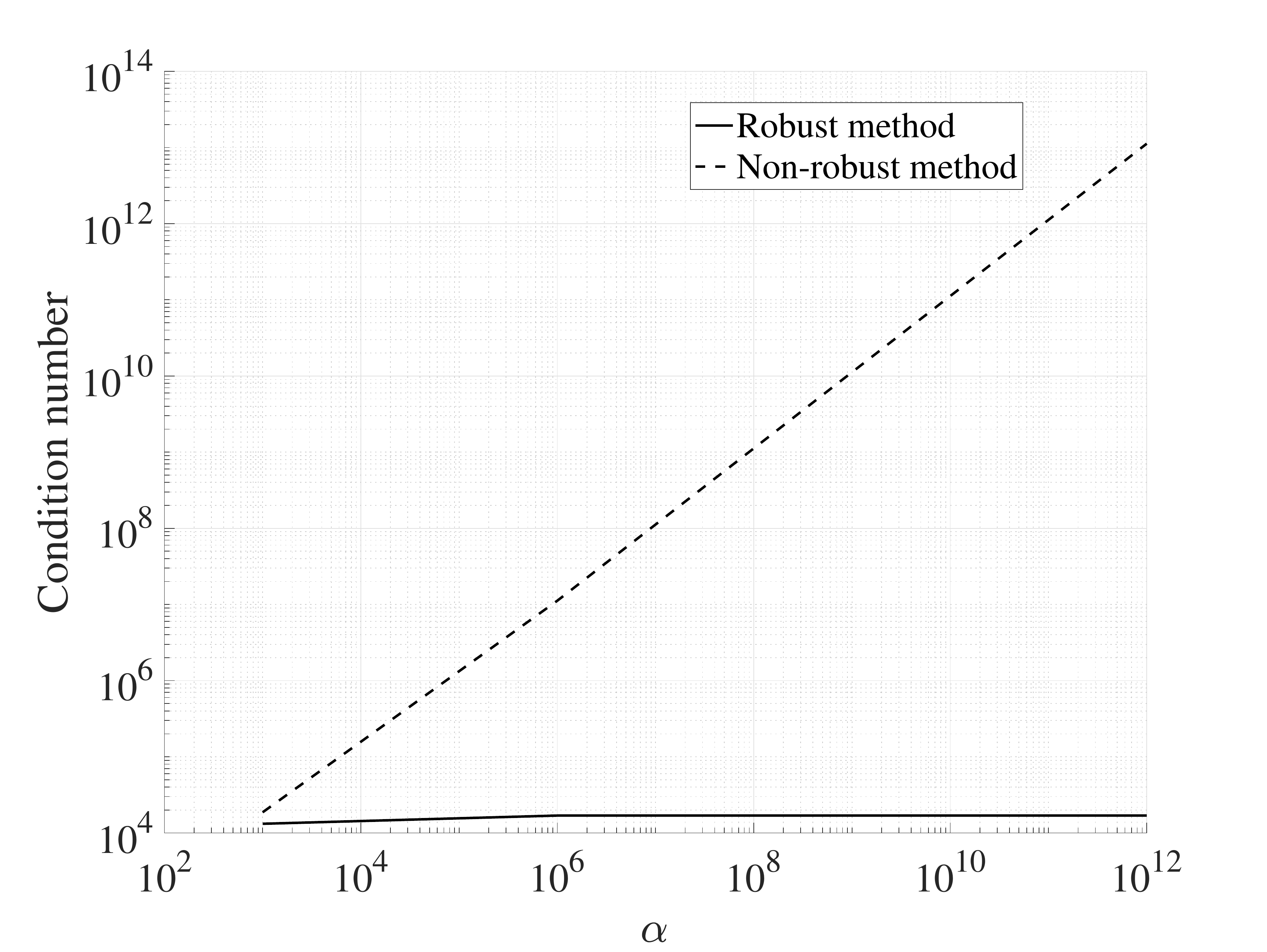}
	\end{center}
	\caption{Left: Condition number as a function of meshsize. Dashed line has inclination 1:2. Right: condition numbers on a fixed mesh with varying $\alpha$ using the robust method (\ref{eq:robustFEM}) and the non-robust method (\ref{eq:FEM}).}
	\label{fig:condest}
\end{figure}

\begin{figure}[ht]
	\begin{center}
		\includegraphics[scale=0.16]{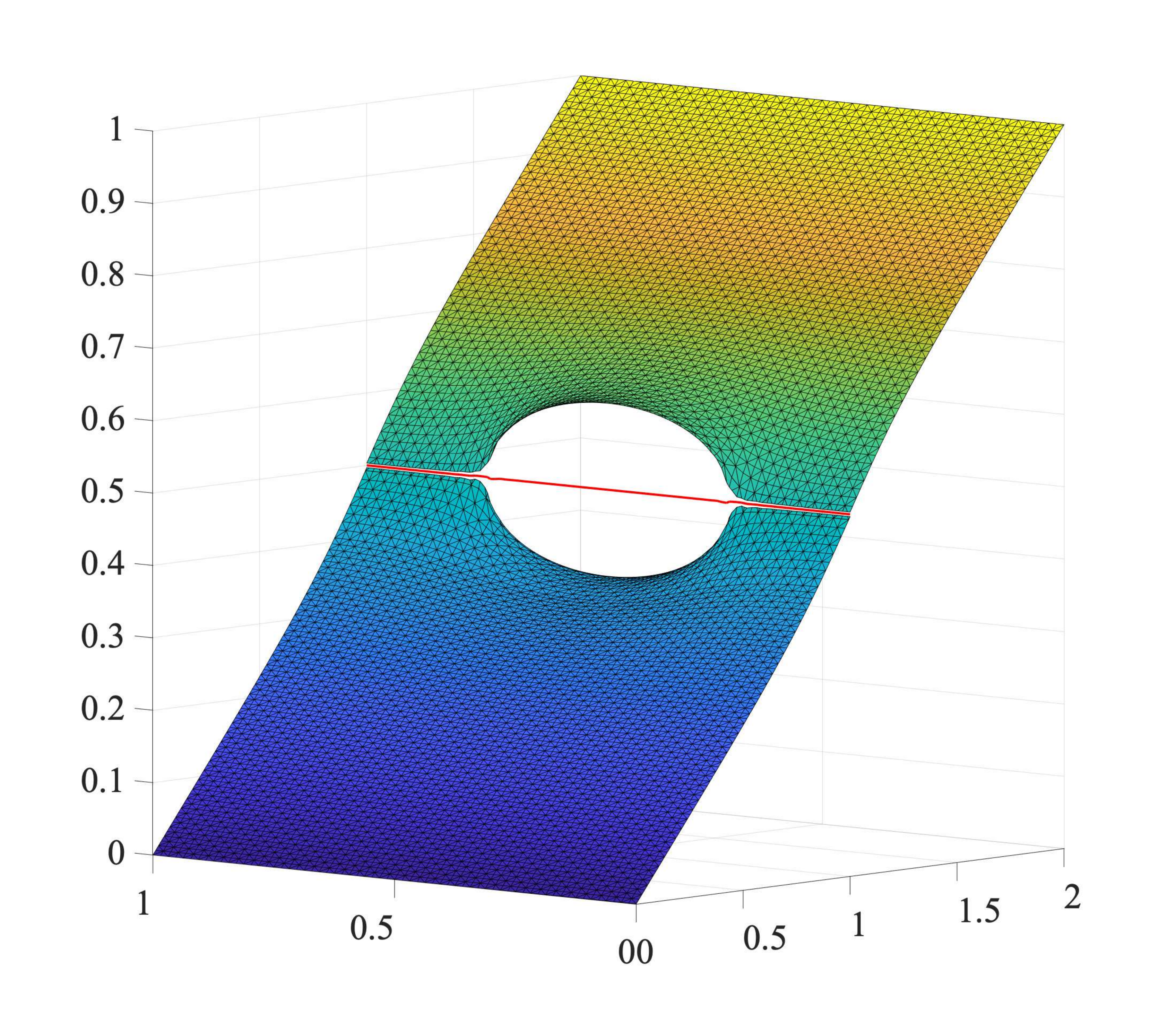}\includegraphics[scale=0.16]{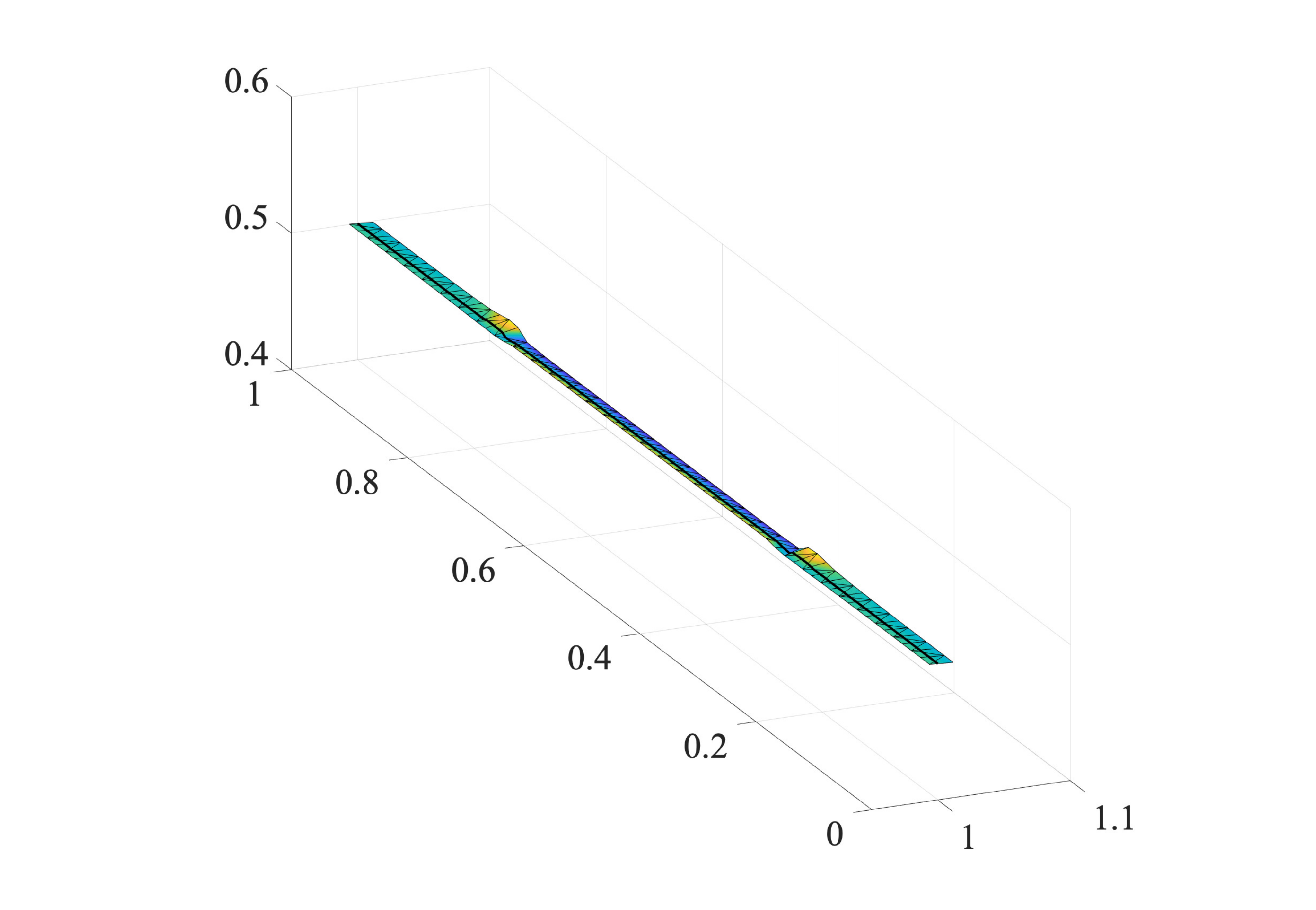}
	\end{center}
	\caption{Elevation of the solution on $\Omega$ and the band containing $\Gamma$ for $\gamma = 0$.}
	\label{fig:elevationzero}
\end{figure}

\begin{figure}[ht]
	\begin{center}
		\includegraphics[scale=0.16]{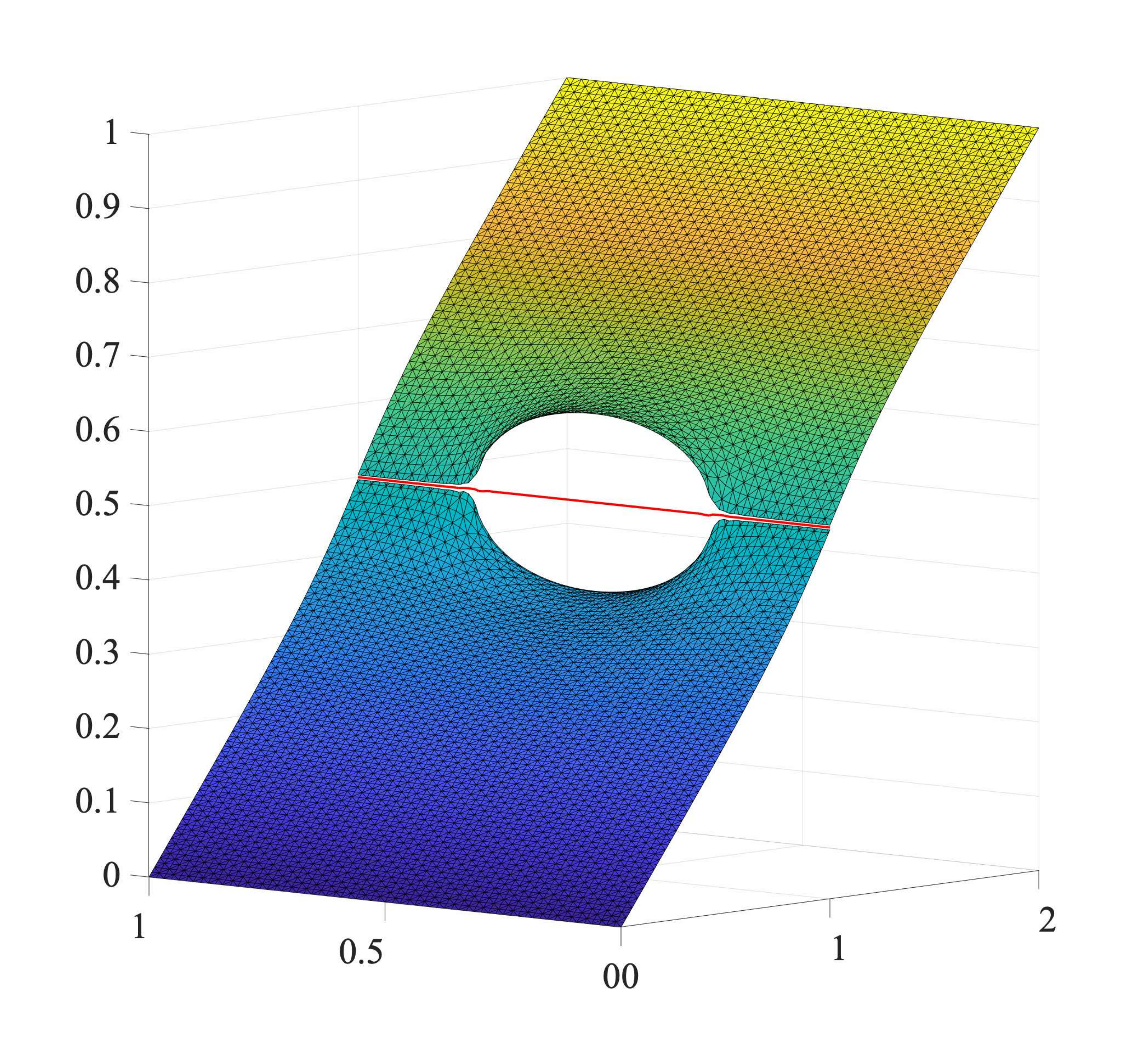}\includegraphics[scale=0.16]{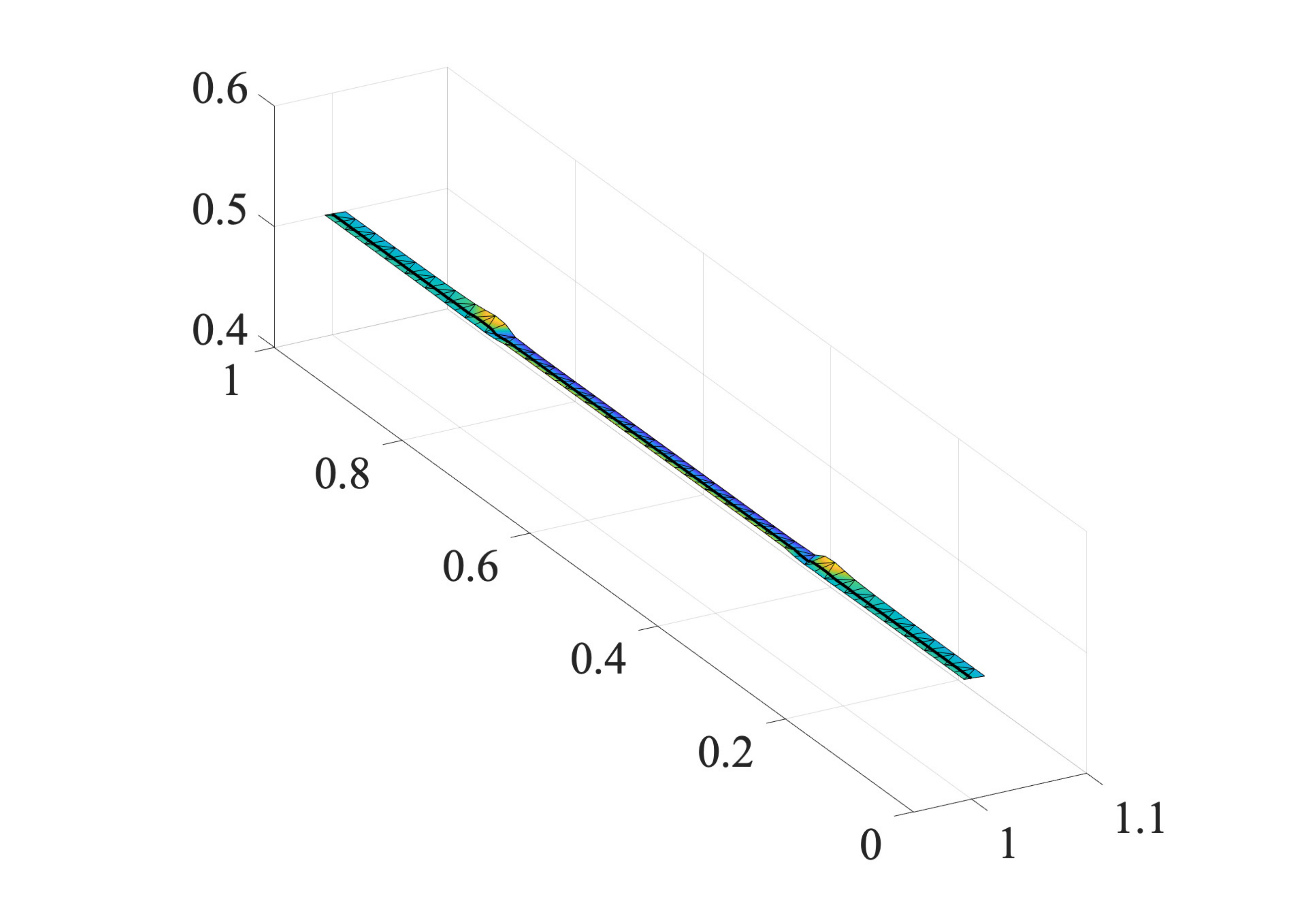}
	\end{center}
	\caption{Elevation of the solution on $\Omega$ and the band containing $\Gamma$ for $\gamma = 10^{-2}$.}
	\label{fig:elevationminus2}
\end{figure}

\begin{figure}[ht]
	\begin{center}
		\includegraphics[scale=0.16]{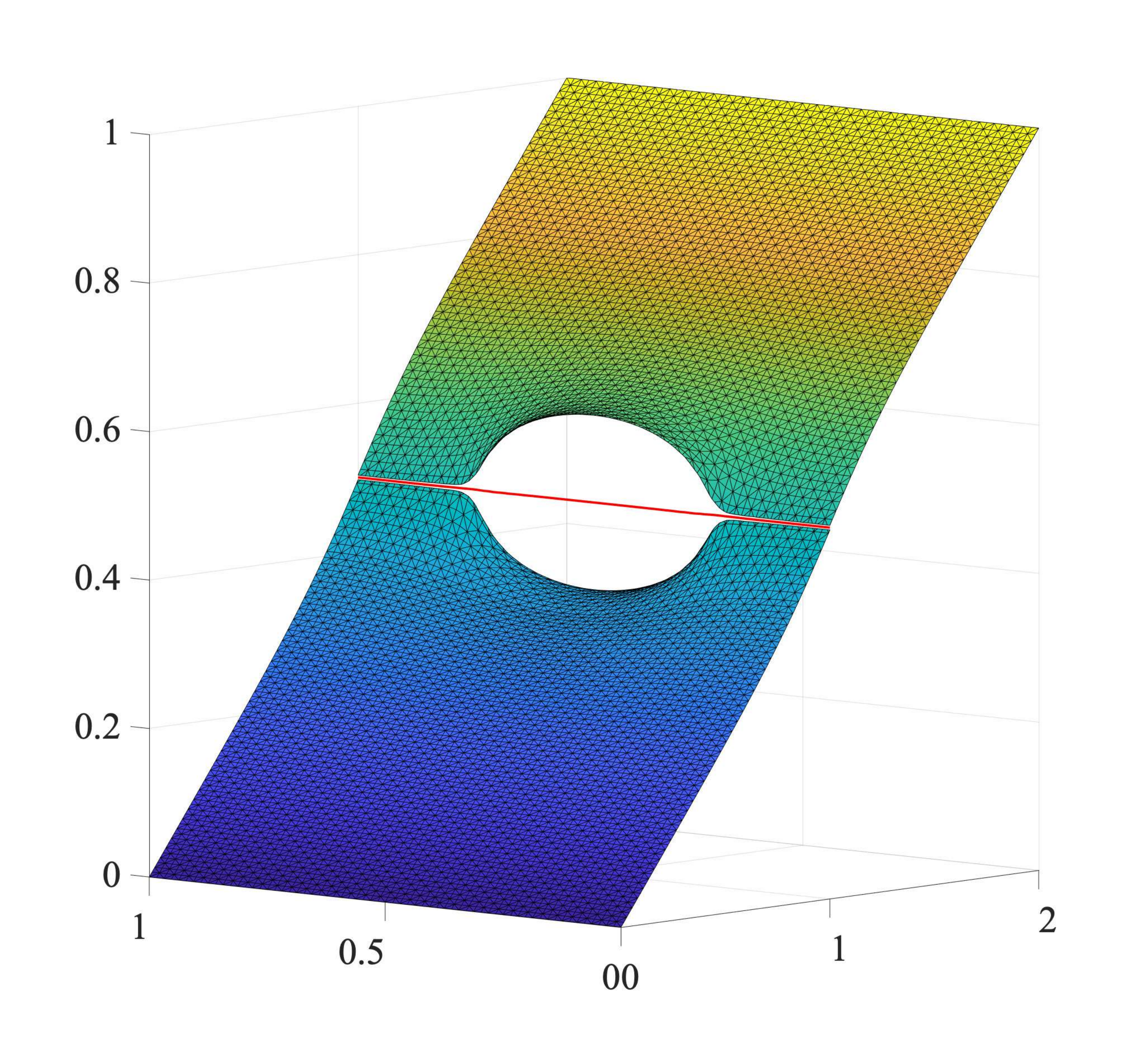}\includegraphics[scale=0.16]{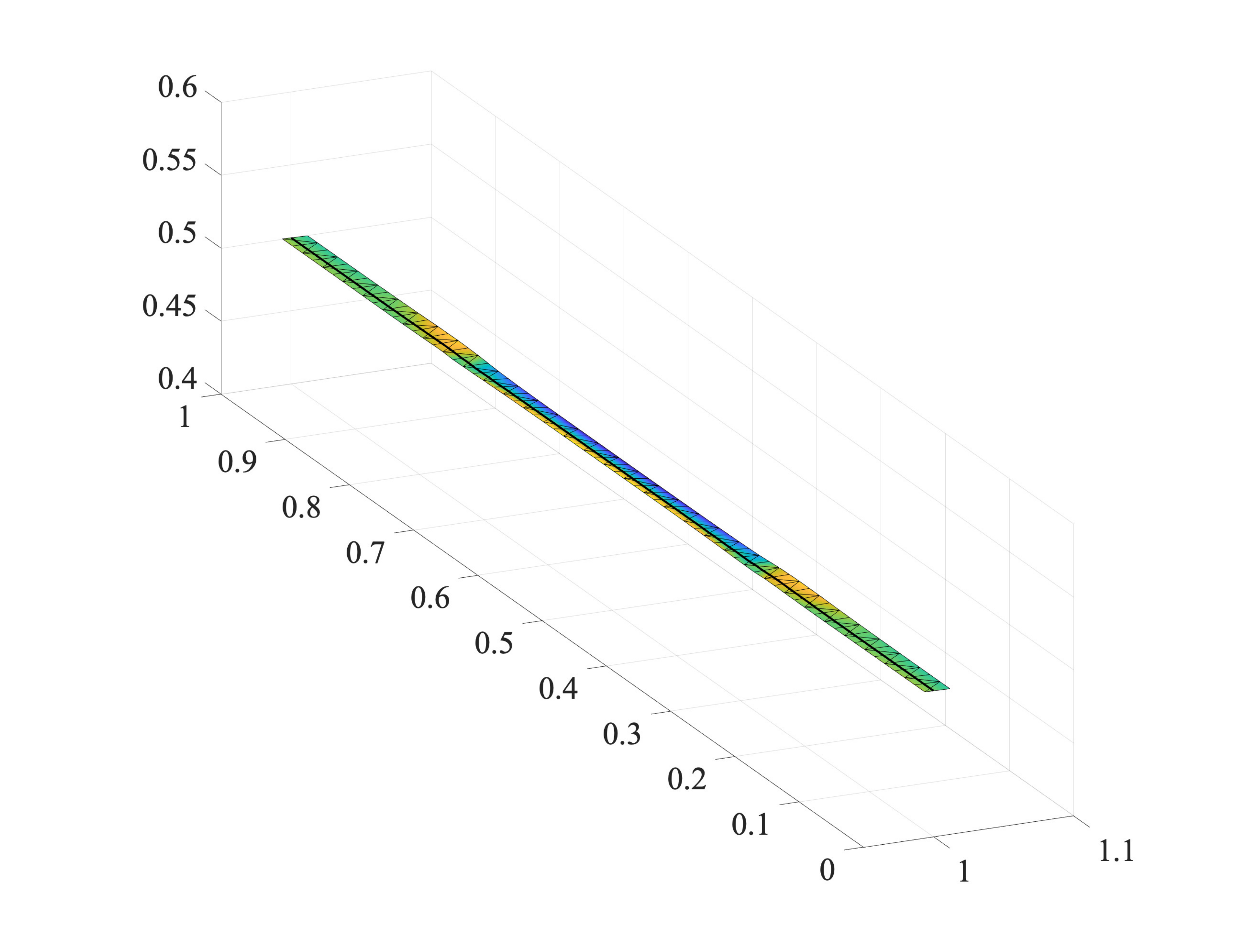}
	\end{center}
	\caption{Elevation of the solution on $\Omega$ and the band containing $\Gamma$ for $\gamma = 1$.}
	\label{fig:elevationone}
\end{figure}

\begin{figure}[ht]
	\begin{center}
		\includegraphics[scale=0.20]{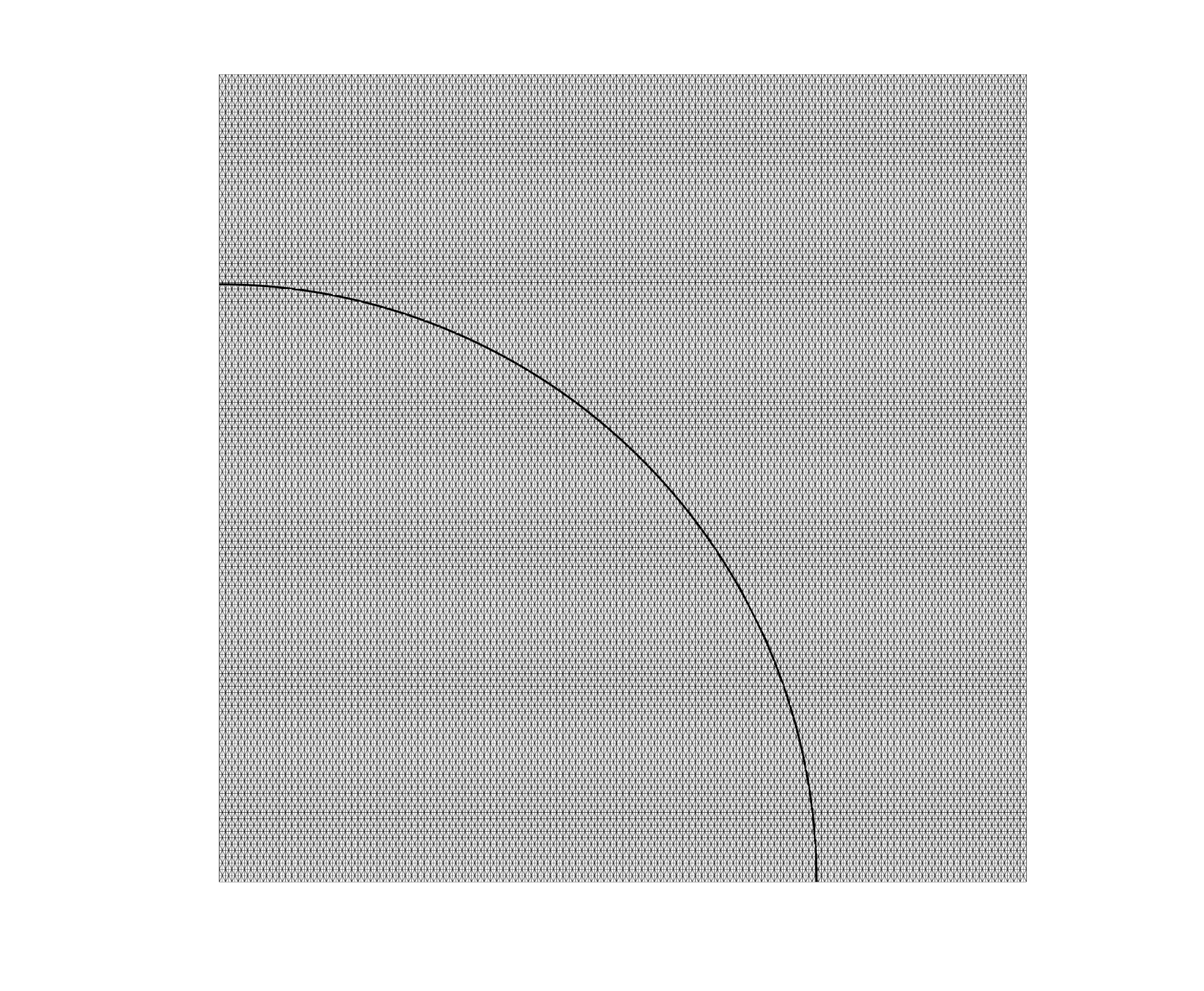}
	\end{center}
	\caption{Computational mesh with interface indicated.}
	\label{fig:mesh}
\end{figure}

\begin{figure}[ht]
	\begin{center}
		\includegraphics[scale=0.18]{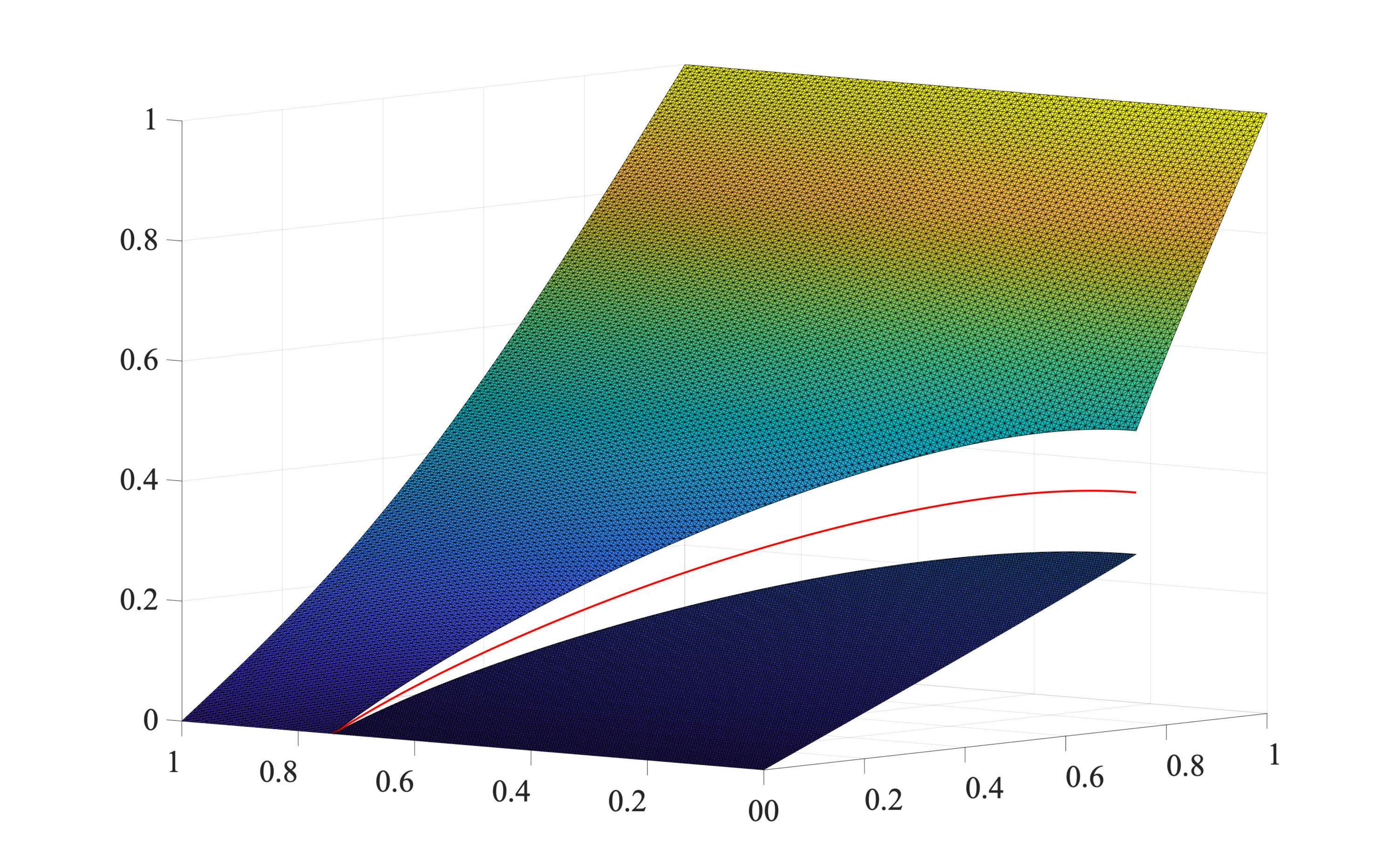}
	\end{center}
	\caption{Elevation for $d=10^{-2}$.}
	\label{fig:dminus2}
\end{figure}

\begin{figure}[ht]
	\begin{center}
		\includegraphics[scale=0.16]{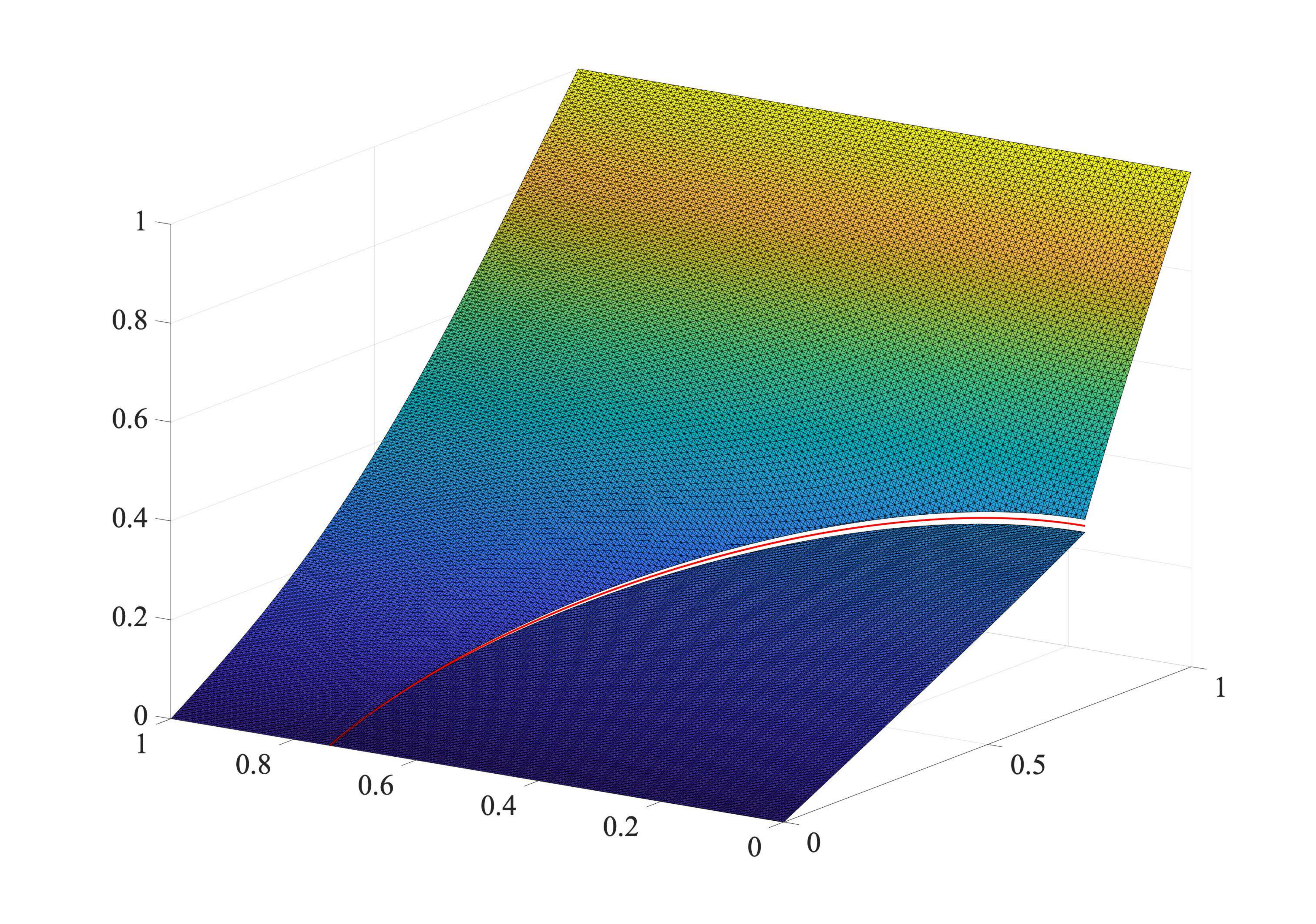}
	\end{center}
	\caption{Elevation for $d=10^{-3}$.}
	\label{fig:dminus3}
\end{figure}

\begin{figure}[ht]
	\begin{center}
		\includegraphics[scale=0.16]{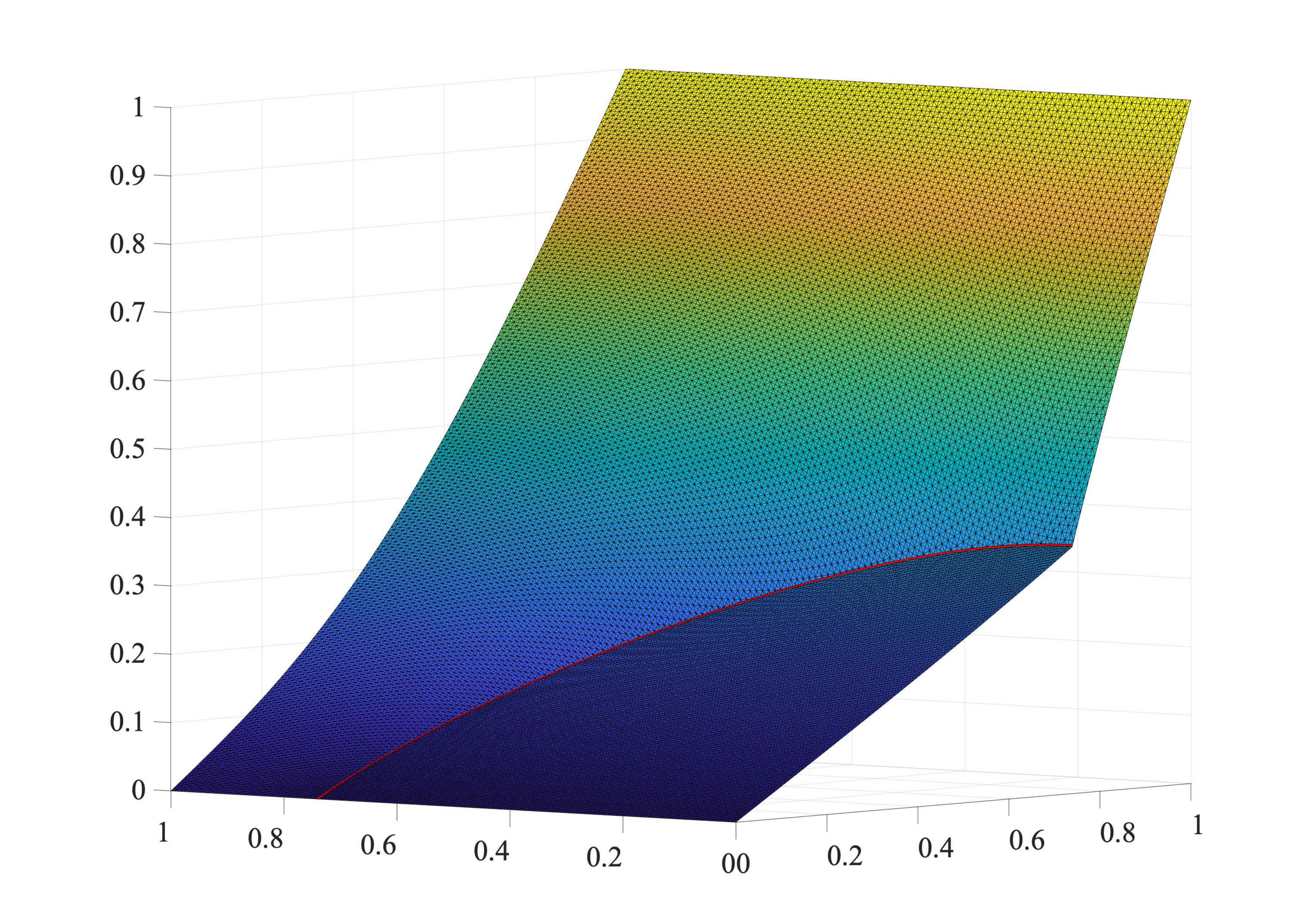}
	\end{center}
	\caption{Elevation for $d=10^{-4}$.}
	\label{fig:dminus4}
\end{figure}

\end{document}